\documentclass[11pt, fleqn, reqno]{amsart}%
\usepackage{epsfig}

\usepackage{bbm}
\usepackage{bm}
\usepackage{enumerate}
\usepackage{amstext}
\usepackage{mathrsfs}
\usepackage{xspace}
\usepackage{amsfonts}
\usepackage{amsmath}
\usepackage{amstext}
\usepackage{amsthm}       
\usepackage{amssymb}
\usepackage{xspace}
\usepackage[active]{srcltx}
\usepackage{cite}
\usepackage{mathtools,mathrsfs,graphicx,appendix,cancel}
\usepackage{xspace,enumerate,comment,bbm,nth}
\usepackage{hyperref}

\usepackage{tikz,pgfplots}

\newcommand{\CC}{\D{C}}

\newcommand{\R}{\mathbb R}

\newcommand{\Lip}{\D{Lip}}
\newcommand{\N}{\mathbb N}
\newcommand{\parts}{\mathscr{P}}

\newcommand{\T}{\mathbb T}
\newcommand{\Z}{\mathbb Z}

\newcommand{\sol}{\mathfrak{S}}

\newcommand{\comp}{\mbox{\scriptsize  $\circ$}}
\newcommand{\eps}{\varepsilon}
\newcommand{\ucv}{\rightrightarrows}

\newcommand{\MM}{{\R}}
\newcommand{\TM}{{\R\times\R}}
\newcommand{\TTorus}{{\T^1\times\R}}

\newcommand{\omegasetG}{\omega\big(\{u^\lambda_G\}\big)}
\newcommand{\gammaxlambda}{\gamma_x^\lambda}

\newcommand{\1}{\mathbbm1_{0}}

\newcommand{\supp}{\operatorname{supp}}

\newcommand{\A}{\mathcal{A}}
\newcommand{\Bor}{\mathscr{B}}

\newcommand{\e}{\textrm{\rm e}}
\newcommand{\E}{\mathcal{E}}

\newcommand{\Mis}{\mathfrak{M}}

\newcommand{\tagliato}{$\kern-5 mm -$}
\newcommand{\tagliat}{$\kern-4 mm -$}

\newcommand{\cchi}{\mbox{\large $\chi$}}
\newcommand{\abra}[1]{(\ref{#1})}

\newcommand{\D}[1]{\mbox{\rm #1}}
\newcommand{\dd}{\D{d}}

\newcommand{\diam}[1]{\operatorname{diam}(#1)}

\newcommand{\weakcv}{\rightharpoonup}

\newtheorem{teorema}{Theorem}
\newtheorem{prop}[teorema]{Proposition}
\newtheorem{lemma}[teorema]{Lemma}
\newtheorem{definition}[teorema]{Definition}
\newtheorem{cor}[teorema]{Corollary}

\newtheorem{guess}[teorema]{Remark}
\newtheorem{example}[teorema]{Example}

\newenvironment{oss}{\begin{guess} \begin{rm}}{\end{rm} \end{guess}}
\newenvironment{definizione}{\begin{definition} \begin{rm}}{\end{rm}
\end{definition}}
\begin{document}

\title{On the vanishing discount approximation\\ for compactly supported perturbations\\
of periodic Hamiltonians: the 1d case}\footnote{\rm}

\author{Italo Capuzzo Dolcetta \and Andrea Davini 
}
\address{Dip. di Matematica, {Sapienza} Universit\`a di Roma,
P.le Aldo Moro 2, 00185 Roma, Italy}
\email{capuzzo@mat.uniroma1.it}
\email{davini@mat.uniroma1.it}
\date{February 15, 2023}
\begin{abstract} 
We study the asymptotic behavior of the viscosity solutions $u^\lambda_G$ of 
the Hamilton-Jacobi (HJ) equation
\begin{equation*}
\lambda u(x)+G(x,u')=c(G)\qquad\hbox{in $\MM$}
\end{equation*}
as the positive discount factor $\lambda$ tends to 0, where $G(x,p):=H(x,p)-V(x)$ is the 
perturbation of a 
Hamiltonian $H\in\D C(\TM)$, 
$\Z$--periodic in the space variable and convex and coercive 
in the momentum, by a compactly supported potential $V\in\D{C}_c(\R)$. 
The constant $c(G)$ appearing above is defined as the 
infimum of values $a\in\R$ for which the 
HJ equation $G(x,u')=a$ in $\MM$ admits bounded viscosity subsolutions.
We prove that the functions $u^\lambda_G$ locally 
uniformly converge, for $\lambda\rightarrow 0^+$, to a specific solution 
$u_G^0$ of the critical equation
\begin{equation}\label{abs}\tag{*}
G(x,u')=c(G)\qquad\hbox{in $\MM$}.
\end{equation}
We identify $u^0_G$ in terms of projected Mather measures for $G$ and 
of the limit $u^0_H$ to the unperturbed periodic problem. This can be regarded as an extension to a noncompact setting of 
the main results in \cite{DFIZ1}. Our work also includes a qualitative analysis of 
\eqref{abs} with a weak KAM theoretic flavor. 
\end{abstract}

\maketitle
\section*{Introduction}
In this paper we study the asymptotic behavior  
of the viscosity solutions $u^\lambda_G$ of discounted Hamilton-Jacobi (HJ) 
equations of the form
\begin{equation}\label{intro eq G discounted}
\lambda u(x)+G(x,d_x u)=c(G)\qquad\hbox{in $M$}
\end{equation}
as the discount factor $\lambda$ tends to 0. Here we consider the case where $M$ is the 1-dimensional 
Euclidean space $\R$, and $G:=H-V$ is the perturbation of a 
$\Z$--periodic (in the space variable) Hamiltonian $H\in\D C(\TM)$ via a 
compactly supported 
potential $V\in\D{C}_c(\R)$. The Hamiltonian $H$ is furthermore assumed convex 
and coercive 
in the 
momentum variable $d_xu$. The constant $c(G)$ appearing above is defined as the 
infimum of values $a\in\R$ for which the 
HJ equation $G(x,d_x u)=a$ in $M$ admits {\em bounded} viscosity subsolutions.

It turns out that the functions $u^\lambda_G$ are equi--Lipschitz and locally equi--bounded 
in $\MM$,  
hence, by the 
Ascoli-Arzel\`a Theorem and the stability of the notion of viscosity solution, they 
converge, {\em along subsequences} as $\lambda$ goes to $0$, to 
viscosity solutions of the {\em critical equation}
\begin{equation}\label{intro eq G critical}
G(x,d_x u)=c(G)\qquad\hbox{in $M$.}
\end{equation}
Since the above critical equation has infinitely many solutions, even up to additive constants in general, 
it is not clear at this level that the limit of the $u^\lambda_G$ along different 
subsequences is the same. 
The main theorem we prove is the following.

\begin{teorema}\label{teo1 intro}
Let $M:=\MM$ and let $G(x,p)=H(x,p)-V(x)$ be as above. For every $\lambda>0$, let us denote by 
$u^\lambda_G$ the unique bounded viscosity 
solution of the discounted Hamilton-Jacobi equation \eqref{intro eq G discounted}.   
Then the whole family $\{u^\lambda_G\}_{\lambda>0}$ converges, locally uniformly in $\MM$ as $\lambda\to 0^+$, to a 
distinguished solution $u^0_G$ of the critical equation \eqref{intro eq G critical}.
\end{teorema}

%

Since solutions to equations \eqref{intro eq G discounted} and \eqref{intro eq G 
critical} are equi-Lipschitz, the function $H$ can be possibly modified for $p$ 
outside a large ball without affecting the analysis. 
We shall therefore assume, without any loss in generality, the Hamiltonians to 
be superlinear in $p$.   
In that case, we can introduce the conjugated Lagrangians via  Fenchel's 
formula 
and make use of tools issued from weak KAM Theory, suitably adapted to the case 
at issue, to 
characterize the limit function $u^0_G$. Our output can be summarized as 
follows.  
Let us denote by $c_f(G)$ the {\em free critical value}, defined as the minimum 
value $a\in\R$ 
for which the Hamilton--Jacobi equation $G(x,d_x u)=a$ in $M$ admits {\em possibly unbounded }viscosity 
subsolutions, cf.  \cite{fatmad}.  
Let ${\sol_b}_-(G)$ be the family of bounded  subsolutions  $v:\MM\to 
\R$ of the critical equation \eqref{intro eq G critical} such that 
\begin{equation}\label{intro condition u0}
\int_\TM v(y)\,d\tilde\mu(y,q)\leqslant \1\big(c(G)-c_f(G)\big) \qquad\text{for every 
$\mu\in\tilde\Mis(G)$}, 
\end{equation}
where $\tilde\Mis(G)$ denotes the set of  Mather measures for $G$ and $\1$ 
the indicator function of the set $\{0\}$ in the sense of convex analysis, i.e. 
$\1(t)=0$ if $t=0$ and $\1(t)=+\infty$ otherwise. 
The constraint \eqref{intro condition u0} is meant to be empty whenever $c(G)>c_f(G)$.
%
%
%
%

In what follows, we will denote by $u^0_H$ the uniform limit of the functions 
$u^\lambda_H$ on $\MM$ as $\lambda\to 0^+$, which is well-defined according to the results in \cite{DFIZ1}, and by   
\[
p^-_H(x):=\min\{p\in\R\,:\, H(x,p)\leqslant c(H)\,\},
\qquad
p^+_H(x):=\max\{p\in\R\,:\, H(x,p)\leqslant c(H)\,\}.
\]
We have the following characterization. 

\begin{teorema}\label{teo2 intro} The limit $u^0_G$ of the discounted solutions $u^\lambda_G$ is characterized as follows. 
\begin{itemize}
\item[\bf(I)] If $c(G)>c(H)$, then $c(G)=c_f(G)$ and 
\[
u^0_G(x):=\sup_{v\in{\sol_b}_-(G) }v(x),\qquad{x\in\R}.
\]
Furthermore, $u^0_G$ is coercive on $\MM$.\smallskip\\
\item[\bf(II)] if $c(G)=c(H)>c_f(H)$, then:\smallskip
\begin{itemize}
\item[\bf(A)] if $\int_0^1 p^+_H(x) dx=0$, then for every $x\in\MM$ 
\begin{equation*}
u^0_G(x):=\sup\{v(x)\,:\,v\in{\sol_b}_-(G),\ v\leqslant u^0_H\quad 
 \hbox{in $(-\infty,\underline y_V)$\,}\},
  \end{equation*}
where $\underline y_V:=\min\left(\supp(V)\right)$\smallskip;
 \item[\bf(B)]if $\int_0^1 p^-_H(x) dx=0$, then for every $x\in\MM$ 
\begin{equation*}
 u^0_G(x):=\sup\{v(x)\,:\,v\in{\sol_b}_-(G),\ v\leqslant u^0_H\quad 
 \hbox{in $(\overline y_V,+\infty)$\,}\},
 \end{equation*}
where $\overline y_V:=\max\left(\supp(V)\right)$.\medskip
\end{itemize}
In either case, $u^0_G$ is bounded on $\MM$.\smallskip\\
\item[\bf(III)] if $c(G)=c(H)=c_f(H)$, then $c(G)=c_f(G)$ and 
\[
u^0_G(x):=\sup_{v\in{\sol_{b}}_-(G) }v(x).
\]
Furthermore, $u^0_G$ is bounded on $\MM$.
\end{itemize}
\end{teorema}

The model example we have in mind is   
$H(x,p):=|\theta+p|-U(x)$, where $\theta\in\R$ and $U$ is a 1-periodic 
function. Here, $c_f(H)=-\min_\MM U$, $c_f(G)=-\min_\MM (U+V)$ and clearly 
$c_f(G)\geqslant c_f(H)$ (although this is a general fact, see Proposition \ref{prop c(G)>c(H)}). Furthermore, 
there exists $\theta^-\leqslant 0 \leqslant \theta^+$ such that 
\[
c(G)=c(H)\quad\hbox{if $\theta\in\R\setminus (\theta^-,\theta^+)$,\qquad and\qquad} 
c(G)> c(H)\quad\hbox{if $\theta\in (\theta^-,\theta^+)$.}
\]
When $c_f(G)=c_f(H)$, then $\theta^-=\theta^+=0$ and only cases 
(II) and (III) above occur. This happens, for instance, when the potential $V$ is nonnegative. 
When $c_f(G)>c_f(H)$, then 
$(\theta^-,\theta^+)$ is nonempty and case (I) occurs when $\theta\in (\theta^-,\theta^+)$, 
while case (II) occurs when $\theta\in\R\setminus (\theta^-,\theta^+)$.\smallskip

The interest to analyze the asymptotic behavior as $\lambda\to 0^+$ of solutions to equation \eqref{intro eq G discounted} 
has roots, as it is well-known, in homogenization theory, see \cite{LPV } and also \cite{CDI01, Ev92}, 
as well as in ergodic optimal control, see for example \cite{AL98, bardi, C89, G00}. 

The present work can be seen as an extension to a noncompact setting of the 
study carried out in 
\cite{DFIZ1}, where Theorem \ref{teo1 intro} is proved in the case $G=H$ and 
$M:=\R^d$ for 
any $d\geqslant 1$. 
In this instance, the constant $c(H)$, as defined above, agrees with
the one considered in \cite{DFIZ1}, known as 
{\em ergodic constant} in the PDE literature, or {\em Ma\~ne critical value} in 
the context of weak KAM Theory. It can be characterized by the property of being 
the unique 
constant $a\in\R$ for which the equation \ \ $H(x,d_x u)=a$\ \  admits 
$\Z^d$-periodic viscosity solutions on $\R^d$, or, equivalently, viscosity 
solutions 
on the torus $\T^d$. 
The analysis in \cite{DFIZ1} is carried out by making use of tools from weak KAM 
Theory in the 
spirit of \cite{IS11}. 
The breakthrough brought in by \cite{DFIZ1} consists in pointing out the crucial role 
played in the asymptotic study by Mather measures and by a distinguished family 
of probability 
measures associated with the discounted equations. 
These kinds of results have been subsequently generalized in many different 
directions, 
such as, for instance,  HJ equations with a possibly degenerate second-order term 
\cite{MiTra17, IMTa17,IMTb17}, weakly 
coupled systems of HJ 
equations \cite{DZ21, IsLi20, Is21}, mean field games \cite{CaPo19}, 
HJ equations with vanishing negative discount 
factor \cite{DaLi21}, contact-type HJ equations \cite{CCIsZh19, Chen21, Z20}.  
In all these papers, the base manifold $M$ is assumed compact (typically,  
the $d$-dimensional flat torus $\T^d$), which is a crucial assumption in order 
to have 
pre-compactness of the family of measures introduced for the study. 

Few results are available when $M$ is noncompact. In this direction, the first convergence result 
appeared in literature is, to the best of our knowledge, the work  \cite{CDMe88}. 
Here, $M:=\R^d$ and $G(x,p):=g(x)\cdot p-f(x)$, 
where $f$ is a bounded, uniformly continuous function and $g$ is a 
Lipschitz, bounded and strongly monotone vector field on $\R^d$. This latter condition 
yields that the dynamical system induced by $g$ has a unique fixed point $x_0$, which is 
an attractor for the dynamics. This means that the associated critical equation 
\eqref{intro eq G critical} has a unique solution up to additive constants. 
With the language of weak KAM Theory, which was formalized and introduced by A. Fathi 
around ten years later, we may say that $\{x_0\}$ is the Aubry set for 
\eqref{intro eq G critical} in this setting.

The vanishing discount problem for HJ equations in the whole $d$-dimensional Euclidean space 
has been recently considered in \cite{IsSic20}, where the authors study the asymptotic of 
solutions to an equation of the form 
\eqref{intro eq G discounted} with $M:=\R^d$ and where 
$c(G)$ is replaced by $c_f(G)$.  
The Hamiltonian $G$ is a continuous function of general form, which is only assumed 
convex and coercive in the momentum variable $p$, locally uniformly in $x$. 
The asymptotic convergence is proved under a condition (see (A3) or its weaker version 
(A3$'$) in \cite{IsSic20}) that implies the existence of a bounded subsolution to 
\eqref{intro eq G critical} which is uniformly strict outside some compact set of $\R^d$.  
In the basic example  
$G(x,p):=|p|-f(x)$,\  these 
conditions boil down to requiring that the infimum of $f$ 
is not attained at infinity. As in \cite{DFIZ1}, the asymptotic behavior of discounted solutions is derived from 
weak convergence of suitably associated measures, but the way these are introduced is rather different. 
Condition (A3)/(A3$'$) in \cite{IsSic20} is used to prove tightness of these families of measures, 
thus recovering the necessary compactness which is no longer guaranteed by the ambient space. 

Even though the setting and the point of view is not the same, 
this study is similar in spirit with case (I) here (cf. Theorem \ref{teo2 intro}). 
Indeed, case (I) is characterized by the existence of a bounded subsolution to 
\eqref{intro eq G critical} which is uniformly strict outside some compact set 
containing the support of $V$. This implies tightness of the families of discounted probability measures 
introduced for the asymptotic analysis, 
which are defined, in accordance with \cite{DFIZ1}, as occupational measures on optimal curves for the discounted solutions. 
We are again in a situation where,  even though the ambient space is 
noncompact, there is no dispersion of mass at infinity, so the strategy followed in  
\cite{DFIZ1} can be implemented.
It is worth pointing out that the dimension 1 of the ambient space plays no role 
for this part of the work. 
Indeed, the proofs of Theorem 1 and Theorem 2 in case (I) can be easily adapted to the case 
$M:=\R^d$ for any $d\geqslant 1$. 
In this regard, we point out that the definitions of the objects we work with, 
most notably the critical values $c(G)$ and $c_f(G)$, the associated intrinsic semidistances 
and the free Aubry set, are given in dimension 1 for uniformity of notation, but 
they make sense in $\R^d$ for any $d\geqslant 1$. We have also taken care of writing 
proofs of the facts needed in this part of the analysis in such a way to be easily 
adapted to any space dimension. 

We also remark that case (I) herein considered has some intersection with 
the work \cite{IsSic20}. 
For instance, when $H(x,p):=|p|-U(x)$, we have $c(H)=c_f(H)=-\min_\MM U$ and 
$c(G)=c_f(G)=-\min_{\MM}(U+V)$, so case (I) occurs if and only if 
$\min_{\MM}(U+V)<\min_\MM U$, namely if and only if 
$G$ satisfies condition (A3)/(A3$'$) in \cite{IsSic20}.

The novelty of our analysis consists in dealing with cases where a bounded subsolution 
to 
\eqref{intro eq G critical}, uniformly strict outside some compact set, does not 
exist, cf. 
case (II) and (III) in Theorem \ref{teo2 intro} above and 
Theorem \ref{teo no bounded critical sol}.  In this situation, the probability measures 
$\tilde\mu^\lambda_x$ associated with the discounted equations may and actually 
do present dispersion of mass at infinity as $\lambda\to 0^+$. We deal with this lack of compactness 
by showing that any such measure 
$\tilde\mu^\lambda_x$ can be written as a convex combination of a probability 
measure $\tilde\mu^\lambda_{x,1}$ whose support is contained in a ball large 
enough so to contain the support of $V$, and a probability measure 
$\tilde\mu^\lambda_{x,2}$ having support contained in the complementary set. By 
standard compactness, the measures  $\tilde\mu^\lambda_{x,1}$ possibly converge 
to a Mather measure $\tilde\mu_{x,1}$ associated 
with $G$ as $\lambda\to 0^+$, and this can happen only when $c(G)=c_f(G)$. On the 
other hand, the measures 
$\tilde\mu^\lambda_{x,2}$ do not see the perturbation of $H$ through the 
potential $V$, so we can exploit the compactness hidden in the model by showing 
that their projections on the tangent bundle of the 1-dimensional torus $\T^1$ 
converge, as $\lambda\to 0^+$, to Mather measures $\tilde\mu_{x,2}$ for $H$. 
{ Even though we believe 
this phenomenon should hold in any space dimension, we are able to prove it 
only in dimension 1. The crucial point consists in showing that the limit measures 
$\tilde\mu_{x,1}$ and $\tilde\mu_{x,2}$ are closed. For this, we use a purely 1-dimensional 
property, namely the fact that the optimal curves for the discounted solutions 
can be taken to be monotone, see Theorem \ref{teo optimal curves}. 
This implies 
that any such optimal curve cannot visit again a ball once it has left it. 
%
Dimension 1 comes also crucially into play in other parts of our analysis for cases (II) and (III) and hence 
it cannot be generalized as it is to higher dimensions. }
For the sake of completeness, we must specify that Theorems \ref{teo1 intro} and 
\ref{teo2 intro} are proved in case (III) under a condition on the asymptotic 
limit $u^0_H$ of the periodic problem,  see condition \eqref{condition u_0} in 
Section \ref{sez asymptotics III}, which is known to hold when $H$ is Tonelli. We have some 
partial results showing that condition \eqref{condition u_0} keeps holding when $H$ is merely 
continuous, but this is not a general fact, cf. Proposition \ref{prop3 u_0} and  Remark \ref{oss3 
u_0}. 

Our work also includes a qualitative study of the critical equation 
\eqref{intro eq G critical} with a weak KAM theoretic flavor, 
which serves as a preliminary step for the asymptotic analysis. 
The definition of the constant 
$c(G)$ is essentially borrowed from the literature on homogenization of HJ equations, 
but its use for the analysis herein addressed seems new. Indeed, 
equations of the form $G(x,d_xu)=c$ in $M$ with $M$ noncompact are usually 
studied in the weak KAM theoretic literature for $c=c_f(G)$. Notice that the value of 
$c_f(G)$ is not affected by shifts of $G$ in the momentum variable, while that of  
$c(G)$ is. Indeed, when homogenization occurs, for instance when $G=H$, the constant 
$c(G)$ is equal to the effective 
Hamiltonian associated with $G$ evaluated at $p=0$. 

When $c(G)>c(H)$, i.e. in case (I), our analysis reveals that 
equation \eqref{intro eq G critical} does not admit 
bounded solutions, see Theorem \ref{teo no bounded critical sol}. 
Furthermore, we show that there exists a coercive solution to  
\eqref{intro eq G critical}, see Corollary \ref{cor existence u_G}. 
This case is also characterized by the existence of a bounded 
subsolution, which is uniformly strict outside a compact set, cf. Theorems  
\ref{teo no bounded critical sol} and \ref{teo strict v_G}.  
This can be interpreted by saying that there is no Aubry set at infinity. 
These results are not specific 
of dimension 1, in fact they can be extended to any dimension with essentially 
same proofs. We point out that, in this case, the existence of Mather measures 
for $G$ in any space dimension can be derived through Proposition 
\ref{prop discounted measures} (rather than via Theorem \ref{teo G Mather measures}, whose 
statement and proof are specifically 1d).

The case $c(G)=c(H)$, corresponding to cases (II) and (III), is characterized 
by the nonexistence of a bounded subsolution such to be uniformly strict outside a 
compact set, cf. Theorem \ref{teo no bounded critical sol}. This feature is not specific of dimension 1 and can be interpreted by saying that there is an Aubry set hidden at infinity.  Our study reveals that this Aubry set at infinity is intervening in the asymptotic problem.  
We believe it would be interesting {\em per se} extending this qualitative analysis 
to higher dimensions and better understanding the role played by the Aubry set at 
infinity, cf. \cite{C01,IsMi07}.\smallskip   

When this work was about to be posted on the ArXiv e-print repository, the authors learnt of a recent pre-print by M. Bardi and H. Kouhkouh \cite{BK22} where the authors prove, among other things, Theorem \ref{teo1 intro} in the case $M:=\R^d$, $G=|p|^2/2-f(x)$ and $c(G)=-\inf_{\R^d} f$, where $f$ is a bounded, Lipschitz and semi-concave function on $\R^d$ which attains its infimum on $\R^d$. Differently from \cite{IsSic20}, the set of minimizers of $f$ does not need to be compact. 
This study has some intersection with case (III) herein considered, even though it relies on completely different techniques.\smallskip

The paper is organized as follows. Section \ref{sez preliminaries} is devoted to fix the main notation and 
assumptions adopted in the paper and to collect some preliminary facts about critical and discounted HJ equations for a general convex and superlinear Hamiltonian $G\in\D C(\TM)$. In Section \ref{sez critical values}  
we introduce the notion of critical value $c(G)$ and compare it with the ergodic constant $c(H)$ when 
$G$ is the perturbation of a periodic Hamiltonian $H$ by a compactly supported potential. Section \ref{sez Mather measures} contains some 
known facts about Mather measures on a compact manifold $M$, together with an extension of these tools to the 
1-dimensional noncompact case $M:=\MM$. Section \ref{sez discounted measures}  is devoted to the study of 
discounted measures, defined following \cite{DFIZ1}: we have collected here the new ideas and facts needed to 
study their asymptotic behavior in the noncompact setting at issue.  The analysis of cases (I), (II) and (III) from 
Theorem \ref{teo2 intro} above is presented in Sections \ref{sez case I}, \ref{sez case II} and \ref{sez case III}, respectively. Each of these sections is divided in two subsections: the first one is devoted to a preliminary analysis of the critical equations for $G$ and $H$, the second one contains the 
asymptotic analysis  for the discounted equations. More precisely, Theorems \ref{teo1 intro} and \ref{teo2 intro} in cases (I), (II) and (III) 
correspond, respectively, to Theorem \ref{theo main}, Theorem \ref{theo main2} and Theorem \ref{theo main3}.\\

\indent{\textsc{Acknowledgements. $-$}}
This research originates from discussions between the authors and  
Charles Bertucci while the latter was visiting the Department of Mathematics 
at Sapienza University of Rome in October 2017. 
The authors wish to thank Charles Bertucci for the 
time and energy dedicated to our early attempts to attack the problem in general space dimensions. 
This research is partially based upon work supported by the National Science Foundation under Grant No. DMS-1440140 while the second author was in residence at the Mathematical Sciences Research Institute in Berkeley, California, during the Fall 2018 semester. 
During this stay, he benefited from discussions with Gonzalo Contreras, Albert Fathi, Hitoshi 
Ishii and Antonio Siconolfi. The authors are also grateful to Antonio Siconolfi for providing the proof of Proposition 
\ref{prop Sez2 effective ham}, and to Maxime Zavidovique for providing the counterexample contained in the Appendix.
%
%

\numberwithin{teorema}{section}
\numberwithin{equation}{section}

\section{Preliminaries}\label{sez preliminaries}

\subsection{A few notations}\label{sez notation}
A function $g:\R_+\to\R$
will be termed {\em coercive} if $g(h)\to +\infty$ as $h\to +\infty$; 
it will be termed {\em superlinear} if $g(h)/|h|$ is coercive. 

Given a metric space $X$, we will write
$\varphi_n\ucv\varphi$ on $X$ to mean that the sequence of
functions $(\varphi_n)_n$ uniformly converges to $\varphi$ on
compact subsets of $X$. Furthermore, we will denote by
$\D{Lip}(X)$ the family of Lipschitz--continuous real functions
defined on $X$.\smallskip\par

We will denote by $B_r(x_0)$ and $B_r$ the open balls in $\MM$ of
radius $R$ centered at $x_0$ and $0$, respectively. The symbol $|\cdot|$
stands for the Euclidean norm.

With the term {\em curve}, without any further specification, we
refer to an absolutely continuous function from some given
interval $[a,b]$ to $\MM$. 
%
%
\subsection{Viscosity solutions} 
We will consider Hamilton--Jacobi equations of the general form
\begin{equation}\label{general HJ equation}
 F(u(x),x,u')=0\qquad\hbox{in $\MM$},
\end{equation}
where $F\in\CC(\R\times \TM)$. 
The notion of {\em solution, subsolution} and {\em supersolution} of 
\eqref{general HJ equation} adopted in this paper is the one in the {\em 
viscosity sense}, see  \cite{bardi,barles,FathiSurvey}. Solutions, subsolutions 
and supersolutions will be implicitly assumed continuous and the adjective {\em 
viscosity} will be often omitted, with no further specification.

Let us recall that a Lipschitz--continuous subsolution $u$ of \eqref{general HJ equation} is also 
an almost everywhere 
subsolution, i.e. \ $F(u(x),x,u'(x))\leqslant 0$ \ for a.e. $x\in \MM$. The converse is true when $F$ is convex in $p$, see 
\cite{bardi,barles,FathiSurvey,Sic}. 

We will also use the following standard results, see for instance \cite{bardi,barles, 
FathiSurvey, Sic}:

\begin{prop}\label{prop when G convex}
Assume $F\in\CC(\TM)$ is independent of $u$ and such that $F(x,\cdot)$ is convex 
in $\R$ for every $x\in \MM$. Let $u\in\D{C}(\MM)$. The following properties hold:
\begin{itemize}
\item[(i)] if $u$ is the pointwise supremum (respectively, infimum) of a family 
of subsolutions (resp., supersolutions) to \eqref{general HJ equation}, then $u$ 
is a subsolution (resp., supersolution) of \eqref{general HJ 
equation};\smallskip
\item[(ii)] if $u$ is the pointwise infimum of a family of equi--Lipschitz 
subsolutions to \eqref{general HJ equation}, then $u$ is a Lipschitz subsolution 
 of \eqref{general HJ equation};\smallskip
\item[(iii)] if $u$ is a convex combination of a family of equi--Lipschitz 
subsolutions to \eqref{general HJ equation}, then $u$ is a Lipschitz subsolution 
of \eqref{general HJ equation}.
\end{itemize}
\end{prop}

Note that items (ii) and (iii) above require the convexity of $F$ in the 
momentum, while item (i) is actually independent of this condition.

We conclude this section by a standard approximation result, see  \cite{fatmad}, Theorem 8.1, that we shall 
repeatedly use in our analysis. 
\begin{lemma}\label{EncoreUneVariante}
Assume $F\in\CC(\TM)$ such that $F(x,\cdot)$ is convex in $\MM$ for every 
$x\in \MM$, and let $u$ be a Lipschitz subsolution of equation \eqref{general HJ 
equation}. Then, for all $\eps >0$, there exists a smooth function $u_\eps : 
\MM\to \R$ such that 
\[
\|u-u_\eps\|_\infty \leqslant \eps\quad  and \quad F(x,u_\eps'(x))\leqslant \eps\qquad\hbox{for all $x\in \MM$}.
\]
\end{lemma} 

\subsection{Hamilton--Jacobi equations and the free Aubry set}\label{sez 
HJ equations}

Throughout the paper, we will call {\em Hamiltonian} a continuous function 
$G:\TM\to\R$. If not otherwise stated, we will always assume that $G$ 
satisfies the following assumptions:\medskip 
\begin{itemize}
\item[(G1)] (convexity)  for every  $x\in \MM$, the map $p\mapsto G(x,p)$ is  
convex 
on $\R$.
\item[(G2)] (superlinearity) there exist two superlinear functions $\alpha,\beta:\R_+\to\R$ such
     that
     \[
     \alpha\left(|p|\right)\leqslant G(x,p)\leqslant \beta\left(|p|\right)\qquad\hbox{for all
     $(x,p)\in \TM$.}
     \]
\item[(G3)] (uniform continuity) $G\in \D{UC}(\R\times B_R)$  for all $R>0$.\smallskip
\end{itemize}
%

The Hamiltonian $G$ will be termed {\em Tonelli} if it is furthermore of class of $C^2$ on $\TM$ and satisfies 
${\partial^2 G}/{\partial p^2}(x,p)>0$ in $\TM$

\begin{oss}\label{oss G1-G2}
Condition (G2) is equivalent to saying that $G$ is superlinear and
locally bounded in $p$, uniformly with respect to $x$. We
deduce from (G1)
\begin{equation}\label{lippo}
  | G(x,p_1)-G(x,p_2)| \leqslant M_R |p-q| \quad\text{for all
  $x\in\MM$, and $p_1$, $p_2$ in $B_R$},
\end{equation}
where $M_R=\sup\{\, |G(x,p)|\,:\,x \in \R, \,|p|\leqslant R+2\,\},$ which is finite thanks to (G2). 
\end{oss}

To any such Hamiltonian, we can associate a Lagrangian function $L:\TM\to\R$ via  the {\em Fenchel transform} of $G$:
\begin{equation}\label{def L}
L(x,q)=G^*(x,q):=\sup_{p\in\R}\left\{\langle p, q\rangle -
G(x,p)\right\}.
\end{equation}
Such a function $L$ satisfies  (G1), (G2) and (G3) as well, for a suitable pair of superlinear functions $\alpha_*,\beta_*$ in place of $\alpha,\beta$. 
In the sequel we shall often use the so-called {\em Fenchel inequality,} namely
\[
L(x,q) \geqslant \langle p, q\rangle -G(x,p)
\qquad
\hbox {for all $x\in\MM$ and $p,q\in\R$.}
\]
\smallskip    
Next, we recall some preliminary facts about stationary Hamilton--Jacobi 
equations of the form
%
 \begin{equation}\label{eq hja}
G(x,u')=a\qquad\hbox{in $\MM$}
\end{equation}
where $a \in \R$. Note that any given {\rm C}$^1$ function $u:\MM\to \R$ is a subsolution 
(resp.\ supersolution) of $G(x,u')=a$ provided $a \geqslant \max\limits_{x\in 
M}G(x,u')$
\big(resp.\ $a\leqslant \min\limits_{x\in M}G(x,u')$\big). Moreover, the 
coercivity and convexity of $G$
entails the following characterization of viscosity subsolutions of 
\eqref{eq hja}, 
see for instance \cite{barles,FathiSurvey}. 

\begin{prop}\label{COMPBASIQUE} Let $G\in\D{C}(\TM)$ satisfy (G1)-(G2)  
 and $a\in \R$. A continuous function $u$ is a viscosity subsolution $u$ of 
\eqref{eq hja} if and only if it is Lipschitz continuous and satisfies 
\[
G(x,u'(x))\leqslant a\qquad\hbox{for a.e. $x\in \MM$.}
\]
Moreover, the set of viscosity subsolutions of  \eqref{eq hja} is 
equi--Lipschitz,  with a common Lipschitz constant $\kappa_a$ given by
\begin{equation}\label{def kappa}
\kappa_a=\sup\{|p|\,:\, G(x,p)\leqslant a\}.
\end{equation}
\end{prop}

We define the {\em free critical value} $c_f(G)$ as
\begin{equation}\label{ManeCrit}
c_f(G)=\inf\{a\in\R\,:\, \text{equation \eqref{eq hja} admits  viscosity
subsolutions}\}.
\end{equation}
The term {\em free} refers to the fact that we are not prescribing a priori any constraint on the 
kind of growth of subsolutions at infinity. Of course, the subsolutions we are considering have at most 
linear growth due to their Lipschitz character.  
The existence of almost everywhere subsolutions for 
\eqref{eq hja} readily implies that the set 
$Z(x,a):=\{p\in\R\,:\,G(x,p)\leqslant a\,\}$ is nonempty for every $x\in\MM$. In 
particular, $c_f(G)$ is greater than or equal to the level of equilibria \ 
$\sup_x\min_p G(x,p)$. 
By the 
Ascoli--Arzel\`a Theorem and the stability of the notion of viscosity 
subsolution, it can be easily proved that the infimum in \eqref{ManeCrit} is attained, 
meaning 
that there are subsolutions also at the free critical level. 

We will  refer to 
\begin{equation}\label{eq free critical}
G(x,u')=c_f(G)\qquad\hbox{in $\MM$}
\end{equation}
as the {\em free critical equation}. 


Following \cite{FS05}, we carry out the study of properties of
subsolutions of \abra{eq hja}, for $a \geqslant c_f(G)$, by means of
the semidistances $S_a$ defined on $\MM\times \MM$ as follows:
    \begin{equation}\label{eq S}
    S_a(x,y)=\inf_{\gamma}\left\{\int_0^1 \sigma_a\big(\gamma(s),\dot\gamma(s)\big)\,\dd
    s\, \right\},
    \end{equation}
where the infimum is taken over all absolutely continuous curves $\gamma:[0,1]\to\MM$ with 
$\gamma(0)=x$ and $\gamma(1)=y$, and $\sigma_a(x,q)$ is the support function of the 
$a$--sublevel
$Z(x,a)$ of $G$, namely
    \begin{equation}\label{sigma}
    \sigma_a(x,q):=\sup\left\{\langle q,p\rangle\,:\,p\in Z(x,a)
    \,\right\}.
    \end{equation}

\noindent There exists a positive constant $\kappa_a$ such that the function $S_a$ satisfies the
following properties, for all $x,y,z\in \MM$:
\begin{eqnarray*}
S_a(x,y)\leqslant S_a(x,z)+S_a(z,y),\qquad
S_a(x,y)\leqslant \kappa_a|x-y|
\end{eqnarray*}
Further properties of  $S_a$ which be useful in the sequel are summarized in the statement below, see \cite{FS05}.

\begin{prop}\label{prop S} Let $ a \geqslant c_f(G)$.
\begin{itemize}
    \item[\em(i)] A function $\phi$ is a viscosity subsolution of \abra{eq hja} if and
    only if
\[
\phi(x)-\phi(y)\leqslant S_a(y,x)\qquad\hbox{for all $x,y\in \MM$.}
\]
In particular, all viscosity subsolutions of \eqref{eq hja} are
$\kappa_a$--Lipschitz continuous.\smallskip
    \item[\em(ii)] For any $y\in \MM$, the functions $S_a(y,\cdot)$
and
    $-S_a(\cdot,y)$ are both subsolutions of \abra{eq hja}.\smallskip
    \item[\em(iii)] For any $y\in \MM$
\[
S_a(y,x)=\sup\{v(x)\,:\,\hbox{$v$ is a subsolution to \eqref{eq
hja} with $v(y)=0$}\,\}.
\]
In particular, by maximality, $S_a(y,\cdot)$ is a viscosity
solution of \eqref{eq hja} in $\MM\setminus\{y\}$.
\end{itemize}
\end{prop}

It is also well known that 
equation \eqref{eq hja} admits solutions for every $a\geqslant c_f(G)$, thanks to the noncompact character of the ambient space 
$\R$.\footnote{Pick a diverging sequence $(y_n)_n$. The functions $v_n(\cdot):=S_a(y_n,\cdot) -S_a(y_n,0)$ form a relatively compact sequence in $\D C(\R)$ by the Ascoli-Arzel\`a Theorem. Any cluster point $u$ of $(v_n)_n$ in $\D C(\R)$ solves \eqref{eq hja} in any fixed open ball in $\R$ by property (iii) above and by stability of the notion of viscosity solution.}
\smallskip

Let us now go back to the free critical equation \eqref{eq free critical}. 
We define the {\em free Aubry set} $\A_f(G)$ as
\[
\A_f(G):=\{y\in M\,:\,S_{c_f(G)}(y,\cdot)\ \hbox{is a solution to \eqref{eq free critical}}\,\}.
\]
The free Aubry set is (possibly empty) closed subset of $\R$ characterized by the 
following property.

\begin{prop}\label{prop rigidity Af}
A point $y\in\MM$ belongs to $\A_f(G)$ if and only if  every subsolution 
$v$ to \eqref{eq free critical} satisfies the supersolution test at $y$, i.e.
\begin{equation}\label{eq rigidity Af}
G(y,p)\geqslant c_f(G)\qquad\hbox{for every $p\in D^- v(y)$,}
\end{equation}
where $D^-v(y)$ denotes the set of subdifferentials of $v$ at $y$.
\end{prop}

The free Aubry set $\A_f(G)$ can be also characterized in terms of strict subsolutions. 
We give a definition first.

\begin{definizione}\label{def v strict}
A subsolution $v\in\Lip(\MM)$ to \eqref{eq hja} is said to be {\em strict} in an open set  
$U\subset\R$ if for every open set $V$ compactly contained 
in $U$ there exist a constant $\delta=\delta(V)>0$ such that 
\begin{equation*}
G(x,v'(x))<a-\delta\qquad\hbox{for a.e. $x\in V$.}
\end{equation*}
We will say that $v$ is {\em uniformly strict} in $U$ if the constant 
$\delta$ can be chosen independently of $V$. 
\end{definizione}
The existence of a subsolution to 
\eqref{eq free critical} which is strict in an open set $U$ implies 
$U\cap \A_f(G)=\emptyset$, in view of Proposition \ref{prop rigidity Af}. 
The converse implication is guaranteed by the following Theorem, cf. \cite{FSC1,FS05}. 

\begin{teorema}\label{teo strict subsolution}
There exist a subsolution $v$ to the free critical equation \eqref{eq free critical} which is 
strict and of class $C^\infty$ in $\MM\setminus \A_f(G)$. 
\end{teorema}

Proposition \ref{prop rigidity Af} and Theorem \ref{teo strict subsolution} yield altogether the following characterization: $\A_f(G)$ is 
the minimal closed subset of $\MM$ for which there exists a strict subsolution outside of it. Otherwise stated, the free Aubry set $\A_f(G)$ 
is the set where the obstruction to the existence of strict subsolutions to \eqref{eq free critical} concentrates.  

All the results presented above are actually true in $\R^d$ for any $d\geqslant 1$. The next 
results are instead specific of dimension 1. 

Let us begin by pointing out some properties of the support function $\sigma_{a}$ that we 
shall need in the sequel. Let us denote by $\D{dom}(Z):=\{(x,a)\in\MM\times\R\,:\,Z(x,a)\not=\emptyset\,\}$. 
This set is closed, as it can be easily checked. The following result holds.

\begin{lemma}\label{lemma properties sigma_a}
The function $(x,a,q)\mapsto \sigma_a(x,q)$ is continuous from 
$\D{dom}(Z)\times\R$ to  $\R$. Furthermore, the map 
$a\mapsto\sigma_a(x,q)$ is strictly increasing and the map 
$q\mapsto\sigma_a(x,q)$ is convex 
and positively homogeneus. 
\end{lemma}

\begin{proof}
To show that $(x,a,q)\mapsto \sigma_a(x,q)$ is continuous on $\D{dom}(Z)\times\R$, it suffices to notice that, for 
$(x,a)\in\D{dom}(Z)$, the set $Z(x,a)$ has either nonempty interior or is a singleton in view of (G1). The other assertions easily follows from the definition of $\sigma_a$ and from properties (G1)-(G2).
\end{proof}


\begin{oss}\label{oss properties sigma_a}
For every $(x,a)\in\D{dom}(Z)$,  the set $Z(x,a)$ is an interval of the form 
$Z(x,a)=[p^-(x,a),p^+(x,a)]$.  The functions $p^\pm(x,a)$ are jointly continuous in $\D{dom}(Z)$. 
This is a consequence of the previous Lemma 
since $p^{-}(x,a)=-\sigma_a(x,-1)$ and  $p^{+}(x,a)=\sigma_a(x,1)$. Also notice that $Z(x,a)\not=\emptyset$ 
for every $a\geqslant c_f(G)$ and $x\in\MM$, i.e. $\MM\times [c_f(G),+\infty)\subseteq \D{dom}(Z)$.
\end{oss}

In dimension 1, we have the following explicit characterization of the free critical value and of the free Aubry set associated with $G$. 

\begin{teorema}\label{teo equilibria}
The following holds:
\begin{equation}\label{claim free critical value}
c_f(G)=\sup_{x\in\MM}\min_{p\in\R} G(x,p).
\end{equation}
Furthermore, there exists $v\in\D C^1(\MM)$ satisfying
\[
G(x,v'(x))\leqslant c_f(G)\quad\hbox{for all $x\in\MM$}
\]
which is strict outside the set $\E(G)$ of equilibria of $G$ defined as follows:
\[
\E(G):=\{y\in\MM\,:\,\min_{p\in\R} G(y,p)=c_f(G)\,\}.
\]
In particular,  $\A_f(G)=\E(G)$.
\end{teorema}

We note that the set of equilibria is empty if and only if the supremum in 
\eqref{claim free critical value} is not attained.

\begin{proof}
Let us temporarily denote by $c_0$ the right-hand side term in 
\eqref{claim free critical value}. We already know that $c_f(G)\geqslant c_0$.  
Let us write \ $Z(x,c_0)=[p^-(x),p^+(x)]$.  
In view of Remark \ref{oss properties sigma_a}, we know that  
the functions 
$x\mapsto p^\pm(x)$ are continuous. Furthermore, they satisfy 
$p^-(x)\leqslant p^+(x)$ for all $x\in\MM$, 
with equality holding if and only if $x\in\E(G)$. Let us set 
\[
 v(x):=\int_0^x\dfrac{p^-(z)+p^+(z)}{2}\, dz
\]
It is easily seen that $v$ is a $C^1$ subsolution of equation \eqref{eq hja} with $a=c_0$,  
which is furthermore strict in $\MM\setminus\E(G)$. This implies $c_f(G)=c_0$ and 
$\A_f(G)\subseteq\E(G)$ in view of Proposition \ref{prop rigidity Af}. 
The converse inclusion $\E(G)\subseteq\A_f(G)$ follows from the fact that 
\eqref{eq rigidity Af} is authomatically satisfied at points $y\in\E(G)$. 
\end{proof}

\subsection{Discounted Hamilton-Jacobi equations}\label{Sec1 discounted HJ}
We will be also interested in {\em discounted Hamilton\,--Jacobi equations} of the form   
\begin{equation}\label{eq discounted}
 \lambda u(x)+G(x,u')=a\quad\hbox{in $\MM$},
\end{equation}
where $a\in\R$ and $\lambda>0$ is the discount factor. The following holds.

\begin{prop}\label{prop a.e. subsolution}
Let $G\in\CC(\TM)$ be a Hamiltonian satisfying (G1)-(G2) and 
$\lambda> 0$. Then any subsolution $w$ of \eqref{eq discounted} is 
locally Lipschitz--continuous and satisfies 
\begin{equation}\label{ineq a.e. subsolution}
\lambda w(x)+G(x,w'(x))\leqslant a\quad \text{ for a.e. $x\in \MM$.} 
\end{equation}
Furthermore, if $w$ is bounded from below on $\MM$, then $w$ is Lipschitz--continuous in $\MM$. 
\end{prop}
\begin{proof}
A subsolution $w$ of \eqref{eq discounted} satisfies 
\[
 G(x,w')
 \leqslant a-\lambda w(x) 
 \leqslant 
 \|(a-\lambda w \big)^+\|_{L^\infty(U)}\qquad\hbox{in $U$}
\]
in the viscosity sense for every open and bounded subset $U$ of $\MM$, where we have denoted by 
$(a-\lambda w \big)^+$ the positive part of the function $a-\lambda w$. We infer that 
$w$ is locally Lipschitz continuous by the coercivity of 
$G$, see \cite{barles}. In particular, it satisfies the inequality \eqref{ineq 
a.e. subsolution} at every differentiability point, i.e. almost everywhere by 
Rademacher's theorem. If $w$ is bounded from below, we can take $U=\R$ above and derive that $w$ is globally Lipschitz on $\MM$,  
with Lipschitz constant given by \eqref{def kappa} with $\|(a-\lambda w \big)^+\|_\infty$ in place of $a$. 
\end{proof}

The crucial difference between equation  \eqref{eq hja} and the 
discounted equation \eqref{eq discounted} with $\lambda>0$ is that the latter 
satisfies a strong comparison principle, see for instance 
\cite[Th\'eor\`eme 4.3]{barles}.

\begin{teorema}\label{teo discounted comparison}
Let $G\in\CC(\TM)$ be a Hamiltonian satisfying (G1)-(G3) and $a\in\R$. Let $w,v\in\D C(\R)$ be a bounded super and   
subsolution to  \eqref{eq discounted}, respectively.  
Then $w\geqslant v$ in $\R$.
\end{teorema}

This comparison principle and a standard application of Perron's method yields the following result, see for instance 
\cite[Th\'eor\`eme 2.12]{barles}.

\begin{teorema}\label{teo discounted sol}
Let $G\in\CC(\TM)$ be a Hamiltonian satisfying (G1)-(G3) and $a\in\R$. Then there exists a unique solution $u^\lambda$ to \eqref{eq discounted} in the class of bounded continuous functions on $\R$. 
Furthermore, $u^\lambda$ is Lipschitz continuous. More precisely, $\|u^\lambda\|_\infty\leqslant (M+|a|)/\lambda$ \ with 
$M:=\max\{|\alpha(0)|,|\beta(0)|\}$, and\ $u^\lambda$ is $\kappa_{M}$--Lipschitz continuous, with $\kappa_{M}$ given by \eqref{def kappa}.
\end{teorema}
%

By exploiting the fact that the solution $u^\lambda$ is globally Lipschitz on $\MM$ and  
by arguing as in the proof of \cite[Th\'eor\`eme 4.3]{barles}, we derive the following result, cf. proof of 
\cite[Proposition 2.6]{DFIZ1}.

\begin{prop}\label{prop application comparison}
Let $G\in\CC(\TM)$ be a Hamiltonian satisfying (G1)-(G3) and $a\in\R$. Let $v\leqslant 0$, $w\geqslant 0$ be a sub and a supersolution to \eqref{eq hja}, respectively. Then $v\leqslant u^\lambda \leqslant w$ in $\MM$ for every $\lambda>0$, where $u^\lambda$ is the unique bounded solution to \eqref{eq discounted}. In particular, the family of functions $\{u^\lambda\,:\,\lambda>0\}$ is equi--Lipschitz and locally equi--bounded in $\MM$, hence pre--compact in 
$\CC(\MM)$. 
\end{prop}

In our analysis, we will exploit the following representation formula for the solution 
$u^\lambda$ of the discounted equation \eqref{eq discounted}: 
\begin{equation}\label{representation formula discounted}
 u^\lambda(x)=\inf_{\gamma}\int_{-\infty}^0 \e^{\lambda s}\big[ 
L \big(\gamma(s),\dot\gamma(s)\big)+a\big]\,d s\qquad\hbox{for every $x\in 
\MM$},
\end{equation}
where the infimum is taken over all absolutely continuous curves 
$\gamma:(-\infty,0]\to \MM$,
with $\gamma(0)=x$, see \cite{bardi} for 
more details. Moreover, we will  need the following property concerning 
minimizing curves, see \cite[Appendix 2]{DFIZ1} for a proof: 

\begin{prop}\label{prop calibrating main}
Let $G\in\CC(\TM)$ be a Hamiltonian satisfying (G1)-(G3). 
Let $\lambda>0$ and $x\in \MM$. Then there exists a Lipschitz curve 
$\gamma^\lambda_x:(-\infty,0]\to \MM$ with $\gamma^\lambda_x(0)=x$ such that 
\begin{equation*}
u^\lambda(x)=\e^{-\lambda t} u^\lambda(\gamma^\lambda_x(-t))+\int_{-t}^0 
\e^{\lambda s} 
\big[L \big(\gamma^\lambda_x(s),\dot\gamma^\lambda_x(s)\big)+a\big]\,d 
s\qquad\hbox{for every $t>0$.}
\end{equation*}
In particular
\begin{equation}\label{eq minimizing curve}
u^\lambda(x)=\int_{-\infty}^0 \e^{\lambda s} 
\big[L \big(\gamma^\lambda_x(s),\dot\gamma^\lambda_x(s)\big)+a\big]\,d s.
\end{equation}
Furthermore, there exists a constant $\tilde\kappa>0$, independent of $\lambda$ and $x$, 
such that 
$\|\dot\gamma^\lambda_x\|_\infty\leqslant \tilde\kappa$.

\end{prop}

A curve $\gamma:(-\infty,0]\to \MM$  will 
be termed {\em minimizing} or {\em optimal} for $u^\lambda(x)$ if $\gamma(0)=x$ and it satisfies
\eqref{eq minimizing curve}.

\begin{oss}\label{oss calibrating curves}
The constant $\tilde\kappa>0$ in Proposition \ref{prop calibrating main} actually depends on the 
functions $\alpha,\,\beta$ appearing in condition (G2) only, cf. with the proof of Proposition 6.2 in \cite{DFIZ1} 
and with Theorem \ref{teo discounted sol}. If $\gammaxlambda$ is differentiable at $t\leqslant 0$ 
and $u^\lambda$ if differentiable at $\gammaxlambda(t)$, we furthermore have 
\begin{equation}\label{eq derivative gammaxlambda}
{\dot\gamma}_x^\lambda(t)\in\partial_p G\big(\gammaxlambda(t),(u^\lambda)'(\gammaxlambda(t))\big)
\qquad
\hbox{and}
\qquad
(u^\lambda)'(\gammaxlambda(t))\in\partial_q L\big(\gammaxlambda(t),{\dot\gamma}_x^\lambda(t)\big),
\end{equation}
see for instance the proof of Proposition 5 in \cite{MS17}. In \eqref{eq derivative gammaxlambda}, the symbols 
$\partial_p G(y,p)$ and $\partial_q L(y,q)$ stand for the subdifferentials of the functions $G(y,\cdot)$ and 
$L(y,\cdot)$ in the sense of convex analysis, respectively. The above differential inclusion holds 
with an equality for every $t<0$ 
whenever $G$ is Tonelli, see for instance Proposition  6 in \cite{MS17}.
\end{oss}

We now proceed by proving a monotonicity property of minimizing curves for $u^\lambda$ 
that will play a crucial role in the proofs of the aymptotic result. 

\begin{teorema}\label{teo optimal curves}
Let $G\in\CC(\TM)$ be a Hamiltonian satisfying (G1)-(G3) and let $u^\lambda$ be the unique bounded solution to \eqref{eq discounted} with $\lambda>0$. For every $x\in\MM$, there 
exists a minimizing curve $\gamma:(-\infty,0]\to \MM$ for $u^\lambda(x)$ such that  
$t\mapsto \gamma(t)$ is monotone. 
\end{teorema}

\begin{proof}
We shall prove the assertion in two steps. 
\medskip\\
\noindent{\bf Step 1.} Let us assume $G$ to be Tonelli. Pick a minimizing curve $\gamma:(-\infty,0]\to \MM$ for $u^\lambda(x)$. Clearly, we can assume $\gamma$ non-constant. Hence, there exists $B <0$ such that 
$\gamma(B )\not=\gamma(0)$. Let us assume for definiteness 
$\gamma(B )<\gamma(0)$. 
We aim to show that $t\mapsto \gamma(t)$ is non-decreasing. 
\smallskip\\
{\bf\underline{Claim I}:}\quad $\gamma(t)\leqslant \gamma(B)$\quad for all $t\leqslant B$.
\medskip\\
Indeed, if this were not the case, there would exists $A<B$ such that 
$\gamma(A)>\gamma(B)$. Since $\gamma$ is moving from the point $\gamma(A)$ to the 
point $\gamma(B)$ in the interval $[A,B]$, there must exist a point 
$t_0\in (A,B)$ such that 
\begin{equation}\label{eq2 calibrating}
\gamma(B)<z:=\gamma(t_0)<\min\{\gamma(A),\gamma(0)\}
\qquad
\hbox{and}
\qquad 
\dot\gamma(t_0)<0
\end{equation}
On the other hand, since the curve $\gamma$ is moving from $\gamma(B)$ to $\gamma(0)$ in 
the interval $[B,0]$, there must exist a point $t_1\in (B,0)$ such that 
$z=\gamma(t_1)$ and $\dot\gamma(t_1)\geqslant 0$. 
Now we use the fact that $\gamma$ is optimal for $u^\lambda$, $u^\lambda$ is differentiable at  
$z\in\gamma((-\infty,0))$ and 
\[
 \dot\gamma(t_0)=\partial_p G(z,u'(z))=\dot\gamma(t_1)
\]
to get a contradiction.\medskip\\ 
{\bf\underline{Claim II}:}\quad $\gamma(a)\leqslant \gamma(b)$\quad for all $a<b\leqslant 0$.
\medskip\\
If this were not the case, there would exists $a<b\leqslant 0$ such that \ $\gamma(b)<\gamma(a)$. 
If $\gamma(b)<\gamma(0)$, we can use the same argument used in Claim I with $b$ in place of $B$ 
and $a$ in place of $A$ to get a contradiction. 
Hence, in view of Claim I, we have 
\[
\gamma(a)>\gamma(b)\geqslant \gamma(0)>\gamma(B)
\qquad
\hbox{and}
\qquad
B<a<b.
\]
Then there exists a point $t_0$ of $\gamma$ in $(a,b)$ such that 
\begin{equation*} 
\gamma(b)<z:=\gamma(t_0)<\gamma(a)
\qquad
\hbox{and}
\qquad 
\dot\gamma(t_0)<0.
\end{equation*}
On the other hand, since the curve $\gamma$ is moving from $\gamma(B)$ to $\gamma(a)$ in 
the interval $[B,a]$, there must exist a point $t_1\in (B,a)$ such that 
$z=\gamma(t_1)$ and $\dot\gamma(t_1)\geqslant 0$. 
Now we use the fact that $\gamma$ is optimal for $u^\lambda$, $u^\lambda$ is differentiable at  
$z\in\gamma((-\infty,0)$ and 
\[
 \dot\gamma(t_0)=\partial_p G(z,u'(z))=\dot\gamma(t_1)
\]
to get a contradiction. 
\medskip\\
\noindent{\bf Step 2.} Let us now go back to the case of a Hamiltonian $G$ of general type. Let 
$(\rho_n)_n$  a sequence of standard regularizing kernels on $\TM$ and set 
$G_n(x,p):=(\rho_n*G)(x,p)+|p|^2/n$. The Hamiltonians $G_n$ are Tonelli and satisfy (G2) for a common pair of functions $\alpha,\,\beta$. Let us denote by $u_n^\lambda$ the solutions of the discounted equation \eqref{eq discounted} with $G:=G_n$. The functions $u^\lambda_n$ are equi-bounded and equi-Lipschitz, in view of Theorem \ref{teo discounted sol}, hence they converge, up to extracting a subsequence (not relabeled), to a 
limit function $u^\lambda$. 
Since $G_n\ucv G$ in $\TM$, by the stability of the notion of viscosity solution we derive that $u^\lambda$ is 
the solution of the discounted equation \eqref{eq discounted}. Let us fix $x\in\MM$ and denote by 
$\gamma_n:(-\infty,0]\to\MM$ a minimizing curve for $u_n^\lambda(x)$.  In view of Step 1, these curves are 
monotone. Up to extracting a subsequence, that again we shall not relabel to ease notation, 
we can furthermore assume that these curve all have the same kind of monotonicity, let us say they are all nondecreasing, to fix ideas. In view of Proposition \ref{prop calibrating main} and Remark \ref{oss calibrating curves}, these curves 
$\gamma_n$ are $\tilde\kappa$-Lipschitz for some common constant $\tilde\kappa>0$, hence they converge, up to extracting a subsequence (not relabeled), to a limit curve $\gamma:(-\infty,0]\to\MM$ with $\gamma(0)=x$.  
Clearly, the curve $\gamma$ is nondecreasing. Let us show it is optimal for $u^\lambda(x)$. 
The Lagrangians $L_n$ associated with $G_n$ via \eqref{def L} uniformly converge to the Lagrangian $L$ associated with $G$ on compact subsets of $\TM$. By standard semicontinuity results in the Calculus of Variations, see \cite[Theorem 3.6]{BGH}, for every fixed $t>0$ we get 
\begin{eqnarray*}
u^\lambda(x)-\e^{-\lambda t} u^\lambda(\gamma(-t))
&=&
\lim_{n\to +\infty} u_n^\lambda(x)-\e^{-\lambda t} u_n^\lambda(\gamma_n(-t))
=
\lim_{n\to +\infty} 
\int_{-t}^0 \e^{\lambda s} 
\big[L_n \big(\gamma_n,\dot\gamma_n\big)+a\big]\,d s\\
&\geqslant& 
\int_{-t}^0 \e^{\lambda s} 
\big[L \big(\gamma,\dot\gamma\big)+a\big]\,d s,
\end{eqnarray*}
thus proving the minimality of $\gamma$.
\end{proof}
\begin{oss}
Please note that, when $G$ is a Tonelli Hamiltonian, we have actually shown that monotonicity holds for any minimizing curve 
$\gamma:(-\infty,0]\to \MM$ for $u^\lambda(x)$. 
\end{oss}

\section{Critical values}\label{sez critical values}
Let $G:\TM\to\R$ be a continuous Hamiltonian satisfying (G1)-(G2), and let us 
consider a family of Hamilton--Jacobi equations of the form 
\begin{equation}\label{eq Sez2 HJa}
G(x,u')=a\qquad\hbox{in $\R$,}
\end{equation}
with $a\in\R$. We define the following {\em critical value} associated with 
$G$ as follows:
\begin{equation}\label{Sez2 critical value}
c(G):=\inf\{ a\in\R\,:\, \hbox{there exists a bounded subsolution of \eqref{eq 
Sez2 HJa}}\}.
\end{equation}
We also recall the definition of free critical value associated with  $G$, cf. \cite{FSC1, fatmad}:
\[
c_f(G)=\min\{ a\in\R\,:\, \hbox{there exists a subsolution of \eqref{eq 
Sez2 HJa}}\}.
\]

The following holds.

\begin{prop}\label{prop when a=c(G)}
Assume there exists a bounded supersolution $w\in\Lip(\R)$ to  \eqref{eq Sez2 HJa}. 
Then $a\leqslant c(G)$. In particular, if $w$ is a solution of \eqref{eq Sez2 HJa}, then 
$a=c(G)$.
\end{prop}

\begin{proof}
Assume by contradiction that $a>c(G)$. Then there exists a bounded and Lipschitz function 
$\tilde u$ such that 
\[
H(x,D\tilde u)\leqslant a-2\delta\qquad\hbox{in $\R$.}
\]
Set $u(x):=\tilde u(x)-\eps\sqrt{1+|x|^2}$. For $\eps>0$ small enough 
\[
 H(x,u')<a-\delta\qquad\hbox{in $\R$}.
\]
Furthermore, since 
\[
\lim_{|x|\to +\infty} \frac{w(x)-u(x)}{|x|}=\eps,
\]
we infer that $\lim_{|x|\to +\infty} \big(w(x)-u(x)\big)=+\infty$. In 
particular, there 
exists $x_0\in\R$ such that 
\[
 (w-u)(x_0)=\inf_{\R} (w-u),
\]
i.e. $u$ is a subtangent to $w$ at $x_0$. Being $w$ a supersolution of 
$H\geqslant a$,  by Proposition 4.3 in \cite{CamSic99}, there exists 
$p_0$ belonging to the Clarke generalized gradient $\partial u(x_0)$ of $u$ at $x_0$ 
such that \ $H(x_0,p_0)\geqslant a$. 
On the other hand, $H(x_0,p_0)\leqslant a-\delta$ being $u$ a subsolution of 
$H< a-\delta$, cf.  \cite[Proposition 3]{Sic}, yielding a contradiction. When $w$ is a solution, the converse 
inequality 
$a\geqslant c(G)$ holds as well by definition of $c(G)$.
\end{proof}

In the periodic case, the following characterization of the critical value 
holds:

\begin{prop}\label{prop Sez2 effective ham} Let $H\in\D C(\TM)$ be Hamiltonian satisfying 
(G1)-(G2) which is
$\Z$--periodic in $x$. Then  
\begin{equation}\label{eq effective ham}
c(H)=\min\{ a\in\R\,:\, \hbox{there exist periodic 
subsolutions of \eqref{eq Sez2 HJa}}\}.
\end{equation}
\end{prop}

\begin{proof}
Let us call $c_p(H)$ the infimum of the set appearing at the right-hand side term of 
\eqref{eq effective ham}. It is well known that this infimum is attained, i.e. 
periodic subsolutions exist at level $a=c_p(H)$ as well, see \cite{LPV, FS05}. 
Clearly, $c(H)\leqslant c_p(H)$. 
Let $a>c(H)$ and pick a bounded subsolution $u\in\D{Lip}(\R)$ of \eqref{eq Sez2 HJa} with $G:=H$. Set 
\[
 v(x):=\inf_{z\in\Z} u(x+z),\qquad x\in\R.
\]
It is easily seen that the function $v$ is well defined, Lipschitz continuous 
and $\Z$--periodic.    
By $\Z$--periodicity of $H$ in the state variable, the map $x\mapsto u(x+z)$ 
is a subsolution of \eqref{eq Sez2 HJa} with $G:=H$ for each $z\in\Z$. Since $H$ is convex in 
$p$, we infer from Proposition \ref{prop when G convex}  that $v$ is itself a subsolution of \eqref{eq Sez2 HJa} as an infimum of subsolutions. 
Hence $a>c_p(H)$ and the claim follows since $a>c(H)$ was arbitrarily chosen. 
\end{proof}

Let $H$ be as in the statement of Proposition \ref{prop Sez2 effective ham}. For any fixed 
$\theta\in\R$, let $H_\theta$ be the Hamiltonian defined as 
$H_\theta(x,p):=H(x,\theta+p)$ for all $(x,p)\in\TM$. The {\em effective Hamiltonian} 
$\overline H$ associated with  $H$ 
(also called $\alpha$-function in the context of weak KAM Theory) is the 
function defined as $\overline H(\theta):=c(H_\theta)$ for every $\theta\in\R$. It is 
well known that $\overline H$ satisfies (G1)-(G2), i.e. is convex and superlinear. 
Furthermore, the following holds, see for instance \cite[Sections 6 and 7]{DS09}.
\begin{teorema}\label{teo Sez2 effective ham}
For every $a\geqslant c_f(H)$, let 
\[
 Z^H(x,a):=\{p\in\R\,:\, H(x,p)\leqslant a\,\}=[p^-_H(x,a), p^+_H(x,a)]
 \qquad
 \hbox{for all $x\in\R$.}
\]
Then 
\[
 \overline H(\theta)=\inf\left\{a\geqslant c_f(H)\,:\,\theta\in 
 \Big[\hbox{$\int_0^1 p_H^-(x,a)dx,\int_0^1 p_H^+(x,a)dx$}\Big]\right\}
 \quad\hbox{for all $\theta\in\R$}.
\]
Furthermore, \ $\displaystyle c_f(H)=\min_{\R} \overline H$. 
\end{teorema}

Given $H$ as in the statement of Proposition \ref{prop Sez2 effective ham} 
and $V\in\D{C}_c(\R)$, we define  
\[
G(x,p):=H(x,p)-V(x)\qquad\hbox{for every $(x,p)\in\R\times\R$.}
\]
The following holds.

\begin{prop}\label{prop c(G)>c(H)}
Let $H$ and $G$ be as above. Then $c(G)\geqslant c(H)$ and  $c_f(G)\geqslant c_f(H)$. 
\end{prop}

\begin{proof}
Let us prove $c(G)\geqslant c(H)$. Pick $a>c(G)$ and let 
$u\in\D{Lip}(\R)$ be a bounded subsolution to \eqref{eq Sez2 HJa}. Let $(z_n)_n$ 
be a diverging sequence 
in $\Z$, i.e. $\lim_n |z_n|=+\infty$, and set $u_n(\cdot):=u(z_n+\cdot)-u(z_n)$. 
Each function $u_n$ is a subsolution 
of 
\[
G_n(x,u_n')\leqslant a\qquad\hbox{in $\R$,}
\]
with $G_n(\cdot,\cdot):=G(z_n+\cdot,\cdot)$.  
Since $u$ is bounded and Lipschitz on $\R$, the family of functions 
$(u_n)_{n}$ is equi--bounded and equi--Lipschitz, 
in particular there exists a bounded function $v\in\Lip(\MM)$ such that, 
up to subsequences, $u_n\ucv v$ in $\R$. Now, since 
$G_n\ucv H$ in $\R\times\R$, by the stability of the notion of viscosity 
subsolution we conclude that $v$ satisfies 
\[
 H(x,v')\leqslant a\qquad\hbox{in $\R$.}
\]
This proves that $a\geqslant c(H)$ for every $a>c(G)$, hence $c(G)\geqslant 
c(H)$. 

The proof of the inequality $c_f(G)\geqslant c_f(H)$ goes along the same lines, the only 
difference being that the functions $(u_n)_{n}$ are {\em locally} equi--bounded 
in $\R$ since they are equi--Lipschitz and satisfy $u_n(0)=0$ for all $n\in\N$. 
\end{proof}

The fact that the critical value $c(G)$ is either equal or not equal to $c(H)$ carries the following interesting information.

\begin{teorema}\label{teo no bounded critical sol}
Let us consider the critical equation
\begin{equation}\label{eq no bounded critical sol}
G(x,u')=c(G)\qquad\hbox{in $\MM$}.
\end{equation}
\begin{itemize}
\item[(i)] If $c(G)>c(H)$, equation \eqref{eq no bounded critical sol} does not admit bounded solutions.\smallskip
\item[(ii)] If $c(G)=c(H)$, equation \eqref{eq no bounded critical sol} does not admit a bounded subsolution which 
is uniformly strict outside some compact set $K$.
\end{itemize}
\end{teorema}

\begin{proof}
Let us prove (i). Assume by contradiction that there exists a bounded solution $u$ of 
\eqref{eq no bounded critical sol}. The same argument used in the proof of Proposition 
\ref{prop c(G)>c(H)} shows that there exists a bounded solution $v$ of 
$H(x,v')=c(G)$\ in $\MM$. From Proposition \ref{prop when a=c(G)} we infer that $c(H)=c(G)$, 
reaching a contradiction.

Let us prove (ii). Assume by contradiction that there exists a bounded subsolution $u$  of 
\eqref{eq no bounded critical sol} and a constant $\delta>0$ such that \ $G(x,u'(x))<c(G)-\delta$\ for a.e. $x\in\MM\setminus K$. The same argument used in the proof of Proposition 
\ref{prop c(G)>c(H)} shows that there exists a bounded subsolution $v$ of 
$H(x,v')<c(G)-\delta$\ in $\MM$, contradicting the definition of $c(H)$ since $c(H)=c(G)$.
\end{proof}

\section{Minimizing measures}\label{sez minimizing measures}

\subsection{Mather measures}\label{sez Mather measures}
Let $X$ be a metric separable space. A {\em probability measure} on $X$ is a nonnegative, countably additive set function $\mu$ defined on the $\sigma$--algebra $\Bor(X)$ of Borel subsets of $X$ such that $\mu(X)=1$. 
In this paper, we deal with probability measures defined either on $M$, with $M=\T^1$ or 
$M=\R$, or on its tangent bundle $TM$. A measure on $TM$ is denoted by $\tilde \mu$, where 
the tilde on the top is to keep track of the fact that the measure is on the space $TM$. 
We say that a sequence $(\tilde\mu_n)_{n}$ of probability measures on $TM$ (narrowly) converges 
to a probability measure $\tilde\mu$ on $TM$, in symbols $\tilde\mu_n\weakcv\tilde\mu$, if 
\begin{eqnarray}\label{def weak cv}
\lim_{n\to +\infty}\int_{TM} f(x,q)\,\dd\tilde\mu_n(x,q)
=
\int_{TM} f(x,q)\,\dd\tilde\mu(x,q)\quad\hbox{for every $f\in\D{C}_b(TM)$}, 
\end{eqnarray}
where $\D{C}_b(TM)$ denotes the family of continuous and bounded real functions 
on $TM$. 
If $\tilde\mu$ is a probability measure on $TM$, we denote by $\mu$ its projection 
${\pi_1}_\#\tilde\mu$ on $M$, i.e. the probability measure on $M$ defined as  
\[
 {\pi_1}_\#\tilde\mu(B):=\tilde\mu\big(\pi_1^{-1}(B)\big)\qquad\hbox{for every $B\in\Bor(M)$}.
\]
Note that 
\[
 \int_M f(x)\,{\pi_1}_\#\tilde\mu(x) 
 = 
 \int_{TM} \left( f\comp\pi_1 \right)(x,q)\,\dd \tilde\mu(x,q)
 \qquad\hbox{for every $f\in\D{C}_b(M)$}, 
\]
where $\CC_b(M)$ denotes the family of bounded continuous functions on $M$.

Let $H:\TM\to\R$ be a continuous Hamiltonian satisfying (G1)-(G2) 
and such that $H$ is $\Z$-periodic in $x$. 
Let us denote by $L_H$ the Lagrangian associated with  
 $H$ via the Fenchel transform \eqref{def L}.
Mather theory states that the constant  $-c(H)$, where $c(H)$
 is the critical value, can be also obtained by minimizing the integral of the 
Lagrangian over $TM$ with respect to closed probability measures  
on $TM$. The definition of closed measure is the following:

\begin{definition}\label{def closed measure}
A probability measure $\tilde\mu\in\parts(TM)$ is termed {\em 
closed} if it satisfies the following properties:
\begin{itemize}
\item[(i)] \quad $\displaystyle \int_{TM} |q|\, d\tilde\mu(x,q) 
<+\infty;$\smallskip
\item[(ii)] \quad for every $\varphi\in C_b^1(M)\cap\Lip(M)$
\[
\int_{TM} \varphi'(x)\,q \  d\tilde\mu(x,q)=0.
\]
\end{itemize}
\end{definition}

A way to construct a closed measure is the following: if $\gamma: [a,b]\to M$ is
an absolutely continuous curve, define the probability measure $\tilde \mu_\gamma$ on $TM$, by
\begin{equation}\label{Birkhoff}
\int _{TM}f(x,q) \,d\tilde \mu_\gamma(x,q):=\frac1{b-a}\int_a^b f\big(\gamma(t),\dot\gamma(t)\big)\,dt,
\end{equation}
for every $f\in\D{C}_{b}(TM)$. 
It is easily seen that $\tilde\mu_\gamma$ is a closed measure whenever $\gamma$ is a loop. \smallskip 

The following holds, see for instance \cite[Appendix 1]{DFIZ1} for a proof. 

\begin{teorema}\label{teo H Mather measures}
Let $H:\TM\to\R$ be a continuous Hamiltonian satisfying (G1)-(G2) 
and $\Z$-periodic in $x$. Let us denote by $L_H$ the Lagrangian associated 
with $H$ via the Fenchel transform \eqref{def L}. Then 
\begin{equation}\label{problem minimizing L_H}
 \min_{\tilde{\mu}} \int_{\TTorus} L_H(x,v)\, d \tilde{\mu}(x,v)=-c(H)
\end{equation}
where $\tilde\mu$ varies in the set of closed measures on $\TTorus$.
\end{teorema}

A closed probability measure on $\TTorus$ that solves the minimization problem 
\eqref{problem minimizing L_H} is called {\em Mather measure} for $H$. 
The set of Mather measures for $H$ will be denoted by $\tilde\Mis (H)$.
A {\em projected Mather measure} is a Borel probability measure in $\mu$ on $\T^1$ of the form
$\mu={\pi_1}_\#\tilde\mu$ with $\tilde\mu\in \tilde\Mis (H)$.  
The set of projected Mather measures will be denoted by $\Mis (H)$.

Let us now consider the Hamiltonian $G:\TM\to\R$ obtained by perturbing the periodic 
Hamiltonian $H$ via a potential $V\in \D C_c(\MM$) as follows:
\[
G(x,p):=H(x,p)-V(x)\qquad\hbox{for all $(x,p)\in\TM$.}
\]
Let us denote by $L_G$ the Lagrangian associated with  
$G$ via the Fenchel transform \eqref{def L}. The following holds. 

\begin{teorema}\label{teo G Mather measures}
Let $G(x,p):=H(x,p)-V(x)$ for all $(x,p)\in\TM$, where $V\in\D C_c(\MM)$ and 
$H:\TM\to\R$ is a continuous Hamiltonian satisfying (G1)-(G2) 
and  $\Z$-periodic in $x$. Let us denote by $L_G$ the Lagrangian associated with  
 $G$ via the Fenchel transform \eqref{def L}. Then 
\begin{equation}\label{problem minimizing L_G}
 \min_{\tilde{\mu}} \int_{\TM} L_G(x,q)\, d \tilde{\mu}(x,q)=-c_f(G)
\end{equation}
where $\tilde\mu$ varies in the set of closed measures on $\TM$. 
\end{teorema}

\begin{proof}
By Theorem \ref{teo equilibria}, 
there exists a $C^1$ subsolution $v$ of the equation $G(x,v')=c_f(G)$ in $\R$ which is 
strict outside $\E(G)$. For any closed measure $\tilde\mu$ on $\TM$ we have 
\begin{eqnarray}\label{eqG closed measure}
 \int_{\TM} L_G(x,q)\, d \tilde{\mu}(x,q)
 \geqslant
 \int_{\TM} v'(x)q \, d \tilde{\mu}(x,q) 
 - 
 \int_{\TM} G(x,v'(x)) \, d \tilde{\mu}(x,q)
 \geqslant
 -c_f(G),
\end{eqnarray}
where we have used Fenchel inequality and the fact that the measure $\tilde\mu$ is 
closed. On the other hand, it is easily seen that the set $\E(G)$ of equilibria 
is nonempty and that any measure of the form 
$\tilde\mu:=\delta_{(y,0)}$ with $y\in\E(G)$ is a closed measure on $\TM$ that solves 
the minimization problem \eqref{problem minimizing L_G}. 
\end{proof}

In analogy with the periodic case, we shall call {\em Mather measure} for $G$ a 
closed probability measure on $\TM$ that solves the minimization problem 
\eqref{problem minimizing L_G}. We will denote by $\tilde\Mis (G)$ the set of 
Mather measures for $G$. A {\em projected Mather measure} for $G$ is a Borel probability 
measure 
$\mu$ on $\MM$ of the form
$\mu={\pi_1}_\#\tilde\mu$ with $\tilde\mu\in \tilde\Mis (G)$.  
The set of projected Mather measures will be denoted by $\Mis (G)$.

As a byproduct of the previous computation, we easily derive the following information. 

\begin{cor}\label{cor Mather set=equilibria}
The set $\Mis(G)$ of projected Mather measures for $G$ coincides 
with the closed convex hull of $\{\delta_y\,:\,y\in\E(G)\,\}$, i.e. the set of delta 
Dirac measures concentrated at points of $\E(G)$.
\end{cor}

\begin{proof}
It is enough to show that any projected Mather measure for $G$ has support contained in 
$\E(G)$. 
Let $\mu={\pi_1}_\#\tilde\mu$ for some Mather measure $\tilde\mu$. Since $\tilde\mu$ is 
minimizing, all inequalities in \eqref{eqG closed measure} are equalities. 
This means that 
$L_G(x,q)=v'(x)q-G(x,v'(x))$\ and $G(x,v'(x))=c_f(G)$\ for $\tilde\mu$--a.e. 
$(x,q)\in\supp(\tilde\mu)$, 
namely
\[
x\in\E(G)
\quad
\hbox{and}
\quad 
q\in\partial_p G(x,v'(x))
\qquad\hbox{for $\tilde\mu$--a.e. $(x,q)\in\supp(\tilde\mu)$}.
\]
We conclude that $\supp(\mu)\subseteq \E(G)$. When $G(x,\cdot)$ is differentiable at $v'(x)$, 
we also get $q=0$, since $G(x,v'(x))=\min_p G(x,p)$.
\end{proof}

\subsection{Discounted minimizing measures}\label{sez discounted measures}
In this section we introduce and study the properties of a class of measures associated with  the 
discounted equation \eqref{eq discounted}. We start by considering the case of a general 
continuous Hamiltonian $G$ satisfying (G1)-(G3). For every $\lambda>0$, we 
denote by $u^\lambda$ the solution of equation \eqref{eq discounted}. 
Let us fix $x\in \MM$ and choose a minimizer $\gamma^{\lambda}_x:(-\infty,0]\to \MM$ for 
$u^\lambda(x)$, i.e. a curve satisfying \eqref{eq minimizing curve} with $\gamma_x^\lambda(0)=x$, 
see Proposition \ref{prop calibrating main}. We define a measure 
$\tilde \mu_x^\lambda=\tilde\mu_{\gamma^\lambda_x}^\lambda$ on $\TM$ by setting
\begin{equation}\label{def discounted measure}
\int_{\TM} f(y,q)\, d \tilde{\mu}^{\lambda}_x 
:=
\int_{-\infty}^0 (\e^{\lambda s})' 
f\big(\gamma_x^\lambda(s),\dot\gamma_x^\lambda(s)\big)\, d s
\qquad\hbox{for every $f\in\D{C}_b(\TM)$.}
\end{equation}
Clearly, the measures 
$\{\tilde\mu_x^\lambda\,:\,\lambda>0, x\in\MM\}$ are probability 
measures whose supports are contained in $\MM\times \overline B_{\tilde\kappa}$ 
for some constant $\tilde\kappa>0$, according to Proposition \ref{prop calibrating main}.  
The projection of $\tilde{\mu}^{\lambda}_x$ on 
the base manifold $\MM$ via the map ${\pi_1}:\TM\ni(x,p)\mapsto x\in\MM$ 
is the measure $\mu_x^\lambda$ defined as follows:
\begin{equation}\label{def projected discounted measure}
\int_{\TM} f(x)\, d {\mu}^{\lambda}_x 
=
\int_{-\infty}^0 (\e^{\lambda s})' 
f\big(\gamma_x^\lambda(s))\, d s
\qquad\hbox{for every $f\in\CC_b(\MM)$.}
\end{equation}
%

We record for later use the following result. 

\begin{prop}\label{prop tightness}
Let $v\in\Lip (\R)$ be a subsolution of \eqref{eq hja} such that $v\leqslant 0$ in $\MM$ and 
\[
G(x,v'(x))< a-\delta \qquad \hbox{for a.e. $x\in U$,}
\]
where $U$ is an open subset of $\MM$ and $\delta>0$ is a constant. Let $\mu_x^\lambda$ be the probability measure defined in 
\eqref{def projected discounted measure}. Then 
\[
\mu^\lambda_x(U)\leqslant \frac{\lambda}{\delta}\left(u^\lambda(x)-v(x)\right). 
\]
\end{prop}

\begin{proof}
Let us pick $r>0$ small enough such that the set $U_r:=\{x\in U\,:\, \D{dist}(x,\partial U)>r\,\}$ is nonempty.  
By Urysohn's Lemma, there exists a continuous function $\psi:\MM\to [a-\delta,a]$ such that 
\ \  $\psi\equiv a-\delta$\ \  on\ \  $U_r$\ \ \ and\ \  $\psi\equiv a$\ \  on\ \  $\MM\setminus U_r$.
In particular, the subsolution $v$ satisfies\ $G(x,v'(x))\leqslant \psi(x)\ \hbox{for a.e. $x\in\MM$.}$ 
Fix $\eps>0$ and apply Lemma \ref{EncoreUneVariante} with $u:=v$ and $F(x,p):=G(x,p)-\psi(x)$ to obtain a function $v_\eps\in\Lip(\R)\cap\CC^\infty(\R)$ such that \ $G(x,v'_\eps(x))\leqslant \psi(x)+\eps$\ for every $x\in\MM$.  In particular,
\begin{eqnarray}\label{ineq strict1}
G(x,v'_\eps(x))\leqslant a+\eps\quad\hbox{in $\MM\setminus U_r$},
\qquad
G(x,v'_\eps(x))\leqslant a-\delta+\eps\quad\hbox{in $U_r$.}
\end{eqnarray}
By exploiting Fenchel's inequality and by taking into account \eqref{ineq strict1} we have 
\begin{eqnarray*}
u^\lambda(x)&=&
\int_{-\infty}^0 
\e^{\lambda s} \big[L \big(\gamma^\lambda_x(s),\dot\gamma^\lambda_x(s)\big)+a \big]\,d s\\
&\geqslant&
\int_{-\infty}^0 \e^{\lambda s}\big[ v'_\eps(\gamma^\lambda_x(s))\,\dot\gamma^\lambda_x(s) 
-G(\gamma^\lambda_x(s), v'_\eps(\gamma^\lambda_x(s)) +a \big] \, ds\\
&\geqslant& 
\int_{-\infty}^0 \e^{\lambda s}(v_\eps(\gamma^\lambda_x(s))'\, ds 
+(\delta-\eps) \int_{-\infty}^0 \e^{\lambda s}\cchi_{U_r}(\gamma^\lambda_x(s))\, ds-\frac{\eps}{\lambda}\\
&\geqslant& 
v_\eps(x)-\int_{-\infty}^0 \big(\e^{\lambda s}\big)' v_\eps(\gamma^\lambda_x(s))\,ds 
+\frac{(\delta-\eps)}{\lambda}\,\mu^\lambda_x\big(U_r)-\frac{\eps}{\lambda},
\end{eqnarray*}
where the last inequality follows by an integration by parts and by taking into account that $v\leqslant 0$ in $\MM$. 
By multiplying both sides by $\lambda>0$, sending $\eps\to 0$ and exploiting the fact that 
$v\leqslant 0$ we get
\[
\mu^\lambda_x(U_r)
\leqslant 
\frac\lambda\delta\left( u^\lambda(x)-v(x)+\int_{\MM} v(y)\, d\mu^\lambda_x(y)\right)
\leqslant
\frac\lambda\delta\big(u^\lambda(x)-v(x)\big).
\]
The assertion follows by sending $r\to 0^+$.
\end{proof}

%

Let us now consider the case $G(x,p):=H(x,p)-V(x)$, where $V\in\D C_c(\MM)$ and  
$H$ is a Hamiltonian satisfying (G1)-(G2) which is  $\Z$-periodic in $x$. Let us assume that 
$c(G)=c(H)$ and denote by $c$ this common critical constant.  Let us denote by $u^\lambda_G$ and 
$u^\lambda_H$ the solutions of the discounted equation \eqref{eq discounted} with $a:=c$, associated with  
with Hamiltonians $G$ and $H$, respectively. 
For every fixed  $x\in \MM$ and $\lambda>0$, let us denote by
$\gamma^{\lambda}_x:(-\infty,0]\to \MM$ a 
{\em monotone} minimizer for the variational 
formula related to $u^\lambda_G(x)$, see \eqref{representation formula discounted}, which exists in force of 
Proposition \ref{prop calibrating main} and Theorem \ref{teo optimal curves}. 
Fix $r>0$ big enough so that 
$\{x\}\cup [\underline y_V,\overline y_V]\subseteq B_r$, where 
$\underline y_V:=\min\big(\supp(V)\big)$ and $\overline y_V:=\max\big(\supp(V)\big)$. 
All objects we are about to 
introduce depend on $r$,  yet explicit dependence on $r$ will be omitted to ease notation. 
Set 
\begin{equation}\label{def T_x}
T^\lambda_x:=\sup\left\{t>0\,:\,\gamma^\lambda_x(-t)\in \overline B_r\right\}. 
\end{equation}
Since 
$t\mapsto\gamma^\lambda_x(t)$ is monotone, we have 
$
\gamma_x^\lambda\left((-T^\lambda_x,0]\right)\subseteq \overline B_r.
$ 
When $T^\lambda_x<+\infty$, we furthermore have\ 
$\gammaxlambda(-T^\lambda_x)\in\partial B_r$,\ 
$\gammaxlambda\big((-\infty,-T^\lambda_x)\big)\subseteq\MM\setminus\overline B_r$. 
Set $\theta^\lambda_x:=1-e^{-\lambda T^\lambda_x}\in (0,1]$ and define two probability measures 
$\tilde\mu_{x,1}^\lambda,\,\tilde\mu_{x,2}^\lambda$ on $\TM$ by setting
\begin{equation}
\label{def2 discounted measure1}
\int_{\TM} f(y,q)\, d \tilde{\mu}^{\lambda}_{x,1}(y,q) 
:=
\frac{1}{\theta^\lambda_x}
\int_{-T^\lambda_x}^0 (\e^{\lambda s})' 
f\big(\gamma_x^\lambda(s),\dot\gamma_x^\lambda(s)\big)\, d s,
\end{equation}
and, when $T^\lambda_x<+\infty$, 
\begin{equation}\label{def2 discounted measure2}
\int_{\TM} f(y,q)\, d \tilde{\mu}^{\lambda}_{x,2}(y,q) 
:=
\frac{1}{1-\theta^\lambda_x}
\int^{-T^\lambda_x}_{-\infty} (\e^{\lambda s})' 
f\big(\gamma_x^\lambda(s),\dot\gamma_x^\lambda(s)\big)\, d s
\end{equation}
for all $f\in\D{C}_b(\TM)$. In order to simplify some statements in what follows, 
it is convenient to assume the probability measure 
$\tilde\mu^\lambda_{x,2}$ to be defined also when 
$T_x^\lambda=+\infty$ (and $\theta^\lambda_x=1$). 
Any choice will serve the purpose.  
Hence, we set $\tilde\mu^\lambda_{x,2}:=\delta_{(x+r,0)}$ 
(the Dirac measure concentrated in the 
point $(x+r,0)$) whenever $T_x^\lambda=+\infty$. 

Let us regard $\R$ as the universal covering of $\T^1$ and denote by 
$\pi:\MM\to\T^1$ the projection map. Given a probability measure on $\TM$, 
we define $\tilde\pi_\#\tilde\mu$ as the push--forward of $\tilde\mu$ via the map 
$\tilde\pi:\TM\to\T^1\times\R$ defined as $\tilde\pi(x,p)=(\pi(x),p)$. 
Such a measure 
$\tilde\pi_\#\tilde\mu$ can be equivalently defined as follows:
\begin{equation*}
\int_{\TTorus} f(x,q)\, d(\tilde\pi_\#\tilde\mu)(x,q):=\int_\TM f(x,q)\, d\tilde\mu(x,q)\qquad\hbox{for all $\CC_b(\TTorus)$,}
\end{equation*}
where $\CC_b(\TTorus)$ is also interpreted as the space of functions $f\in\CC_b(\TM)$ that are 
$\Z$--periodic in $x$. Any such measure belongs to the 
space $\parts(\TTorus)$ of probability measures on $\TTorus$.

The following holds:

\begin{prop}\label{prop2 precompactness} 
Let $ \tilde{\mu}^{\lambda}_{x,1},\,  \tilde{\mu}^{\lambda}_{x,2}$ be the probability 
measures defined above for $x\in\MM$ and $\lambda>0$.  
\begin{itemize}
 \item[(i)] The family of probability measures 
 $\{\tilde{\mu}^{\lambda}_{x,1}\,:\,\lambda>0\,\}$ 
 is pre-compact in $\parts(\TM)$.
 \item[(ii)] The family of Borel measures 
 $\{\tilde\pi_\#\tilde\mu^\lambda_{x,2}\,:\,\lambda>0\}$ 
 is pre--compact in \hbox{$\parts(\TTorus)$}.
\end{itemize}
\end{prop}

\begin{proof}
The curves $\{\gammaxlambda\,:\,\lambda>0\}$ are equi-Lipschitz, let us say $\tilde\kappa$-Lipschitz,    in view of Proposition \ref{prop calibrating main}. By definition of $T^\lambda_x$, we have 
\[
\supp(\tilde{\mu}^{\lambda}_{x,1})\subseteq \overline B_r\times\overline B_{\tilde\kappa},
\qquad
\supp(\tilde\pi_\#\tilde{\mu}^{\lambda}_{x,2})\subseteq  \T^1\times\overline B_{\tilde\kappa}
\qquad
\hbox{for all $\lambda>0$.}
\]
The assertion follows from this.
\end{proof}

\begin{prop}\label{prop2 limit measures}
Let $ \tilde{\mu}^{\lambda}_{x,1},  \tilde{\mu}^{\lambda}_{x,2}$ and $\theta^{\lambda}_x$ be, respectively, the probability measures and the constant defined above for $x\in\MM$ and $\lambda>0$. Let $(\lambda_n)_n$ be a sequence such that $\lambda_n\searrow 0$.  
\begin{itemize}
\item[(i)] Suppose that \ 
$\tilde\mu_{1,x}^{\lambda_n}\weakcv \tilde\mu_{1,x}\quad\hbox{in $\parts(\TM)$}$ \ and \  
$\theta_x^{\lambda_n}\to \theta\in (0,1]$. Then $\tilde\mu_{1,x}$ is a closed measure in $\parts(\TM)$. 
Furthermore, $\tilde\mu_{1,x}$ is a Mather measure in $\parts(\TM)$ for $L_G$.
\smallskip 
\item[(ii)] Suppose that \ 
$\tilde\pi_\#\tilde\mu_{2,x}^{\lambda_n}\weakcv \tilde\mu_{2,x}\quad\hbox{in $\parts(\TTorus)$}$ \ and \  
$\theta_x^{\lambda_n}\to \theta\in [0,1)$. Then $\tilde\mu_{2,x}$ is a closed measure in 
$\parts(\TTorus)$. Furthermore $\tilde\mu_{2,x}$ is a Mather measure in $\parts(\TTorus)$ for 
$L_H$.
\end{itemize}
\end{prop}

\begin{proof} Let us first prove that  $\tilde\mu_{1,x}$ is a closed measure in $\parts(\TM)$. For $\varphi\in\D C^{1}_b(\MM)\cap\Lip(\MM)$ we have 
\begin{multline*}
\int_{\TM} \varphi'(y)q\ d\tilde\mu_{1,x}(y,q)
=
\lim_{n\to +\infty}
\int_{\TM} \varphi'(y)q\ d\tilde\mu^{\lambda_n}_{1,x}(y,q)\\
=
\lim_{n\to +\infty}
\frac{1}{\theta^{\lambda_n}_x}
\int_{-T^{\lambda_n}_x}^0 \lambda_n\e^{\lambda_n s} 
\varphi'\big(\gamma_x^{\lambda_n}(s)\big)\,\dot\gamma_x^{\lambda_n}(s)\ d s\\
=
\lim_{n\to +\infty}
\frac{\lambda_n}{\theta^{\lambda_n}_x}\left(\varphi(y)-e^{-\lambda_nT_x^{\lambda_n}}\varphi\big(\gamma_x^{\lambda_n}(-T_x^{\lambda_n})\big) 
-
\int_{-T^{\lambda_n}_x}^0 \big(\e^{\lambda_n s}\big)' 
\varphi\big(\gamma_x^{\lambda_n}(s)\big)\ d s\right)\\
=
\lim_{n\to +\infty}
\frac{\lambda_n}{\theta^{\lambda_n}_x}\left(\varphi(y)-e^{-\lambda_nT_x^{\lambda_n}}\varphi\big(\gamma_x^{\lambda_n}(-T_x^{\lambda_n})\big) \right) 
-
\lambda_n \int_{\TM} \varphi(y)\ d\tilde\mu^{\lambda_n}_{1,x}(y,q),
\end{multline*}
where for the third equality we have used integration by parts. Now 
\begin{eqnarray*}
\left |
\frac{\lambda_n}{\theta^{\lambda_n}_x}\left(\varphi(y)-e^{-\lambda_nT_x^{\lambda_n}}\varphi\big(\gamma_x^{\lambda_n}(-T_x^{\lambda_n})\big)\right) 
\right |
&\leqslant& 
2\frac{\lambda_n}{\theta^{\lambda_n}_x} \|\varphi\|_\infty\to 0,\\
\left |
\lambda_n \int_{\TM} \varphi(y)\ d\tilde\mu^{\lambda_n}_{1,x}(y,q)
\right |
&\leqslant&
\lambda_n \|\varphi\|_\infty\to 0
\end{eqnarray*}
implying the assertion. The proof of the fact that $\tilde\mu_{2,x}$ is a closed measure in 
$\parts(\TTorus)$ goes along the same lines, with the only difference 
that $\varphi$ needs to be chosen in $\D C^1(\T^1)$.  

To prove the rest of the assertion, we begin by noticing that, by the choice of $\gamma^\lambda_x$, the following holds:
\begin{multline*}
\lambda u^\lambda_G(x)=\\
=\theta^\lambda_x \int_{\TM} \big( L_H(y,q)+V(y)+c\big)\, d\tilde\mu^{\lambda}_{1,x}(y,q) 
+ 
(1-\theta^\lambda_x) \int_{\TM} \big(L_G(y,q)+V(y)+c\big)\, d\tilde\mu^{\lambda}_{2,x}(y,q)\\
=\theta^\lambda_x \int_{\TM} \big( L_H(y,q)+V(y)+c\big)\, d\tilde\mu^{\lambda}_{1,x}(y,q) 
+ 
(1-\theta^\lambda_x) \int_{\TTorus} \big(L_H(y,q)+c\big)\, d\big(\tilde\pi_\#\tilde\mu^{\lambda}_{2,x}\big)(y,q),
\end{multline*}
where for the last equality we have taken into account that 
$(\supp(V)\times\R)\cap \supp(\tilde\mu^{\lambda}_{2,x})=\emptyset$ together with the fact that $L_H$ is $\Z$-periodic in 
the state variable. Let us set $\lambda:=\lambda_n$ and send $n\to+\infty$ in the equality above to obtain
\begin{equation*}
0=\theta \int_{\TM} \big( L_H+V+c\big)\, d\tilde\mu_{1,x} 
+ 
(1-\theta) \int_{\TTorus} \big(L_H+c\big)\, d\big(\tilde\pi_\#\tilde\mu_{2,x}\big).
\end{equation*} 
The assertion readily follows from this after noticing that 
\[
\theta\int_{\TM} \big( L_H+V+c\big)\, d\tilde\mu_{1,x}\geqslant 0,
 \qquad 
(1-\theta)\int_{\TTorus} \big(L_H+c\big)\, d\big(\tilde\pi_\#\tilde\mu_{2,x}\big)\geqslant 0
\]
due to the fact that the measures $\tilde\mu_{1,x}$ and $\tilde\mu_{2,x}$ are closed when $\theta\not=0$ and $\theta\not=1$, 
respectively.
\end{proof}

The next proposition will play a crucial role in the proof of the aysmptotic convergence.

\begin{prop}\label{prop step2}
Let $v$ be a bounded subsolution to 
\[
G(x,v')\leqslant c\qquad\hbox{in $\MM$.}
\]
For every  $\lambda>0$ and $x\in\MM$ we have 
\[
 u^\lambda_{G}(x)\geqslant v(x)
 -
 \left(
 \theta^\lambda_x\int_{\TM} v(y)\, d{\tilde\mu}^\lambda_{x,1}(y,q)
 +
 (1-\theta^\lambda_x)\int_{\TM} v(y)\, d{\tilde\mu}^\lambda_{x,2}(y,q)
 \right)
\]
\end{prop}

\begin{proof}
Fix $\eps>0$ and apply Lemma \ref{EncoreUneVariante} with $F:=G-c$, $u:=v$ to obtain a 
function $v_\eps\in\Lip(\MM)\cap\CC^\infty(\MM)$ such that 
\[
\|v-v_\eps\|_\infty<\eps\qquad\hbox{and}\qquad G(x,v_\eps'(x))<c+\eps\quad\hbox{for every $x\in\MM$.}
\]
By making use of Fenchel's inequality and integrating by parts we get
\begin{eqnarray*}
u^\lambda_G(x)
&\geqslant& 
\int_{-\infty}^0 \e^{\lambda s}
\left[  v_\eps'(\gamma_x^\lambda(s))\,\dot\gamma_x^\lambda(s)
-\eps    \right]\, ds 
=
\int_{-\infty}^0 \e^{\lambda s} \frac{d}{ds}v_\eps(\gamma_x^\lambda(s))\, d s -\frac{\eps}{\lambda}\\
&=& 
v_\eps(x)
- 
\int_{-\infty}^0 \big(\e^{\lambda s}\big)' v_\eps(\gamma_x^\lambda(s))\, d s 
-\frac{\eps}{\lambda}\\
&=&
v_\eps(x)
- 
\left(
\int_{-T^\lambda_x}^0 (\e^{\lambda s})' v_\eps(\gamma_x^\lambda(s))\, d s 
+
\int_{-\infty}^{-T^\lambda_x} (\e^{\lambda s})' v_\eps(\gamma_x^\lambda(s))\, d s 
\right)
-\frac{\eps}{\lambda}
.
\end{eqnarray*}
Now we send $\eps\to 0$ and recall the definition of 
$\tilde\mu_{x,1}^\lambda,\, \tilde\mu_{x,2}^\lambda$ 
to get the assertion. 
\end{proof}

\section{The case  $c(G)>c(H)$}\label{sez case I}

In this section we shall prove the asymptotic convergence in the case $c(G)>c(H)$. We start with a qualitative 
study of the critical equation associated with $G$.

\subsection{The critical equation}
Let us consider the critical equation 
\begin{equation}\label{eq G critical}
 G(x,u')=c(G)\qquad\hbox{in $\R$.}
\end{equation}
By definition, we infer that $c(G)\geqslant c_f(G)$. 
Let us denote by $S_G$ the critical semi-distance associated with  
$G$ via \eqref{eq S} with $a=c(G)$.
The set 
$\{S_G(y,\cdot)\,:\, y\in\MM\,\}$, with 
$S_G$ defined according to \eqref{eq S} with $a:=c(G)$, is a family of subsolutions to \eqref{eq G critical}. 
We recall that such subsolutions are $\kappa$--Lipschitz, with $\kappa$ given by \eqref{def kappa} with $a:=c(G)$, but they are not a priori bounded on $\MM$. In fact, the following holds.

\begin{prop}\label{prop S_c coercive}
Let $\tilde K$ be a compact subset of $\MM$. Then there exists a constant $C=C(\tilde K)$ such that 
\[
S_G(y,\cdot)\geqslant C|x|-\frac1C\quad\hbox{on $\MM$}\qquad\hbox{for every $y\in \tilde K$.}
\]
\end{prop}

\begin{proof}
Let $B_r$ be an open ball  centered in $0$ with radius $r$ large enough so that $\tilde K\cup\supp(V)\subset B_r$. 
By the triangular inequality we get \ 
$S_G(y,x)\geqslant S_G(0,x)-S_G(0,y)\geqslant -\kappa|x-y|\ $ 
for all $x\in\MM$, in particular 
\[
S_G(y,x)\geqslant -2\kappa r\qquad\hbox{for all $x\in\overline B_r$ and $y\in \tilde K$.}
\]
Let us pick $x\in\MM\setminus \overline B_r$ and let $\gamma\in\Lip_{y,x}([0,1];\MM)$ such that 
\[
S_G(y,x)+1>\int_0^1\sigma^G_{c(G)}(\gamma,\dot\gamma)\, ds.
\]
Let us denote by $T:=\max\{s\in [0,1]\,:\, \gamma(s)\in\overline B_r\}$ and by $z:=\gamma(T)$. 
Let us denote by $u$ a periodic solution to $H(x,u')=c(H)$ \ \ in $\MM$. By the coercivity of $H$ in $p$, see condition (H2), we 
get that there exists an $R>0$ such that 
\[
Z^H_{c(G)}:=\{p\in\R\,:\, H(x,p)\leqslant c(G)\}\subseteq B_R\qquad\hbox{for every $x\in\MM$.} 
\]
By taking into account Remark \ref{oss G1-G2}, we get 
\[
c(G)-c(H)=H(x,p_1)-H(x,p_2)\leqslant M_R | p_1-p_2|
\]
for every $x\in\MM$, $p_1\in \partial Z^H_{c(G)}$, $p_2\in  \partial Z^H_{c(H)}$. We infer that 
\[
\sigma^H_{c(G)}(x,q)\geqslant \sigma^H_{c(H)}(x,q)+\delta|q|\qquad\hbox{for every $(x,q)\in\TM$,}
\]
with $\delta:=\big(c(G)-c(H)\big)/M_R>0$. 
By the choice of $r$ and $T$ we have
$$
\sigma^G_{c(G)}(\gamma,\dot\gamma)
=
\sigma^H_{c(G)}(\gamma,\dot\gamma)
\geqslant 
\sigma^H_{c(H)}(\gamma,\dot\gamma)+\delta|\dot\gamma|\quad\hbox{in $[T,1]$.} 
$$
We infer 
\begin{eqnarray*}
&&S_G(y,x)+1
>
\int_0^T\sigma^G_{c(G)}(\gamma,\dot\gamma)\, ds+ \int_T^1\sigma^H_{c(G)}(\gamma,\dot\gamma)\, ds\\
&&\quad\geqslant 
S_G(y,z)+\delta\int_T^1 |\dot\gamma|\, ds+u(x)-u(z)
\geqslant 
\delta |x-z|-2(\kappa r+\|u\|_\infty)\\
&&\quad\geqslant \delta|x|-\big(\delta r+2\kappa r+2\|u\|_\infty\big),
\end{eqnarray*}
proving the assertion for $C>0$ small enough. 
\end{proof}

We proceed by showing the following 

\begin{teorema}\label{teo strict v_G}
There exist a compact set $K$ in $\MM$ with 
$\supp(V)\subset \D{int}(K)$ and a bounded subsolution $v_G$ to  \eqref{eq G critical} 
which is uniformly strict outside $K$, i.e. 
\[
G(x,v'_G(x))\leqslant c(G)-\delta\qquad\hbox{for a.e. $x\in\MM\setminus K$}
\]
for some constant $\delta>0$. In particular, up to an additive constant, we can assume $v_G\leqslant 0$ in $\MM$.
\end{teorema}

In view of Proposition \ref{prop rigidity Af}, we infer 

\begin{cor}
Let $K$ be the compact set given by Theorem \ref{teo strict v_G}. Then $\A_f(G)\subseteq  K$. 
\end{cor}

For the proof of Theorem \ref{teo strict v_G}, we will need the following preliminary fact.

\begin{prop}\label{prop existence w}
Let $a\in (c(H),c(G)]$. Then there exist $u_a, v_a\in\Lip(\R)$ such that 
\begin{itemize}
 \item[(i)] $u_a$ is a solution to  
\begin{equation}\label{eq claim1}
 H(x,u'_a)=a\qquad\hbox{in $\R$,} 
\end{equation}
with $u_a(x)\leqslant -C|x|+1/C$ in $\R$ for some constant $C>0$;\medskip
\item[(ii)] $v_a$ is subsolution of  \eqref{eq claim1}  with $v_a(x)\geqslant 
C|x|-1/C$ in $\R$ for some constant $C>0$.
\end{itemize}
\end{prop}

\begin{proof}
Let $\overline H:\R\to\R$ be the effective Hamiltonian associated with $H$. 
Since $\overline H(0)=c(H)<c(G)$, we infer 
that $0\in\D{int}\{\overline H\leqslant a\}$. In particular, for each 
$\xi\in\{1,-1\}$, there exists a unique $P_\xi\in\{t\xi\,:\,t>0\}$ such that 
$\overline H(P_\xi)=a$.  Let $\tilde w_\xi$ be a periodic solution of 
\[
 H(x,P_\xi+\tilde w'_\xi)=a\qquad\hbox{in $\R$}
\]
satisfying $\tilde w_\xi(0)=0$ and set $w_\xi(x):=\tilde w_\xi(x)+ 
P_\xi x$ for every $x\in\R$. Being $w_1,\,w_{-1}$ 
Lipschitz and locally bounded solutions of \eqref{eq 
claim1}, we infer that the functions 
\begin{eqnarray*}
  u_a(x)&:=&\min\{w_1(x),\,w_{-1}(x)\},\quad x\in\R\\
  v_a(x)&:=&\max\{w_1(x),\,w_{-1}(x)\},\quad x\in\R
\end{eqnarray*}
are a solution and a subsolution of \eqref{eq claim1}, respectively, cf. Proposition \ref{prop when G convex}. Moreover, 
since $\min_{\xi\in\{1,-1\}} |P_\xi|>0$, it is easily seen that there 
exists a constant 
$C>0$ such that
\[
 u_a(x)\leqslant -C|x|+\frac1C,\quad v_a(x)\geqslant C|x|-\frac1C 
\qquad\hbox{for every $x\in\R$.}
\]
This proves the assertion.
\end{proof}

\begin{proof}[Proof of Theorem \ref{teo strict v_G}]
Let take $w(\cdot):=S_G(0,\cdot)$. 
By Proposition \ref{prop S_c coercive}, we know that $w$ is a subsolution to \eqref{eq G critical} satisfying  $w (x)\to +\infty$ as $|x|\to +\infty$. 
Pick a periodic solution $u_H$ to \ $H(x,u')=c(H)$\ \ in $\MM$\  and set 
\[
 v_G(x):=\min\{w (x),u_H(x)+k\}\quad x\in\R,
\]
where $k>0$ is chosen large enough so that $v_G\equiv w $ in a open neighborhood 
$U$ of $\supp(V)$. It is easily seen that $v_G$ is a 
bounded subsolution to \eqref{eq G critical}. Furthermore, as $u_H$ is bounded while 
$w $ is coercive, there exists a compact set $K$ 
such that $v_G\equiv u_H+k$ in $\MM\setminus K$. Note that this implies 
$\supp(V)\subseteq U \subset K$. 
The assertion follows by setting $\delta:=c(G)-c(H)$. 
\end{proof}
%

The following holds:
\begin{teorema}\label{teo1 A(G)}
We have that $c(G)=c_f(G)$ and $\A_f(G)\not=\emptyset$. 
\end{teorema}

\begin{proof}
Suppose either $c(G)>c_f(G)$ or $c(G)=c_f(G)$ and $\A_f(G)=\emptyset$. We claim that there exists $v\in\Lip(\MM)$ which is a strict subsolution 
to \eqref{eq G critical}. In the first case, pick an $a\in (c_f(G),c(G))$  and take a susbsolution $v\in\Lip(\MM)$ of \ $G(x,u')=a$ \ in $\MM$, which exists 
by definition of $c_f(G)$. In the second case, the claim follows as a direct application of Theorem \ref{teo strict subsolution}. 

Since $v$ is a priori neither bounded nor uniformly strict in $\MM$, we are going to modify it in order to produce a uniformly strict and 
bounded subsolution to \eqref{eq G critical}. 
To this aim, choose $a\in (c(H), c(G))$, take $v_a$ as in the statement of Proposition \ref{prop existence w}--(ii) and set 
\[
w(x):=\max\{v(x),v_a(x)-k\},\quad x\in\R,
\]
with $k>0$ large enough so that $w\equiv v$ in an open neighborhood $U$ of $\supp(V)$. 
The function $w$ is a  coercive subsolution to 
\eqref{eq G critical}, which is uniformly strict in every bounded open subset of  $\MM$. 
Now take a periodic solution $u$ to $H(x,u')=c(H)$ in $\MM$ and set 
\[
\tilde w(x):=\min\{w(x),u(x)+k\},\quad x\in\R,
\]
with $k>0$ large enough so that $\tilde w\equiv w\equiv v$ in the open neighborhood $U$ of $\supp(V)$. It is easily seen that $\tilde w$ is a bounded subsolution 
to \eqref{eq G critical}, uniformly strict in every open and bounded subset of $\MM$. 
Since $w$ is coercive while $u$ is bounded, $\tilde w$ agrees with 
$u+k$ outside a compact set $\tilde K$ containing $\supp(V)$. Since $H(x,\cdot)\equiv G(x,\cdot)$ for $x\in\MM\setminus\supp(V)$ and $c(H)<c(G)$, 
we infer that there exist a $\delta>0$ such that 
\[
G(x,{\tilde w}'(x))\leqslant c(G)-\delta\qquad\hbox{for a.e. $x\in\MM$,}
\]
thus contradicting the minimality of $c(G)$. 
\end{proof}

%
%
%
%

The results gathered so far yield the following simple consequence:

\begin{cor}\label{cor existence u_G}
There exists a solution $u_G$ to \eqref{eq G critical} such that 
\[
u_G(x)\geqslant C|x|\quad\hbox{on $\MM$}
\]
for some constant $C>0$. In particular,  $u_G\geqslant 0$ in $\MM$.
\end{cor}

\begin{proof}
Pick $y\in\A(G)$ and set $u_G(\cdot):=S_G(y,\cdot)+1/C$, where $C$ is the constant appearing in the statement of 
Proposition \ref{prop S_c coercive}. 
\end{proof}

\subsection{Asymptotic analysis}\label{sez asymptotics I}
Let us denote by $u^\lambda_G$ the (unique) 
bounded solution of the following discounted equation:
\begin{eqnarray}\label{eq1 G discounted}
\lambda u+G(x,u')&=&c(G)\qquad\hbox{in $\R$}\label{eq G 
discount}
\end{eqnarray}
We recall that the critical equation \eqref{eq G critical} possesses a  coercive solution $u_G\geqslant 0$, see Corollary \ref{cor existence u_G}, and 
a bounded subsolution $v_G\leqslant 0$ such that 
\[
G(x,v'_G(x))\leqslant c(G)-\delta\qquad\hbox{for a.e. $x\in\MM\setminus K$,}
\]
for some $\delta>0$ and some compact set $K\supseteq \supp(V)$, see Theorem \ref{teo strict v_G}.  
In view of Proposition \ref{prop application comparison}, we infer that 
the functions $\{u^\lambda_G\,:\,\lambda>0\}$ are equi--Lipschitz and locally equi--bounded on $\MM$, and satisfy 
\begin{equation}\label{bounds discounted solutions}
v_G\leqslant u^\lambda_G\leqslant u_G\quad\hbox{in $\R$}.  
\end{equation}
%
Let us denote by 
\[
\omegasetG:=\left\{u\in\Lip(\MM)\,: \, u^{\lambda_k}_G\ucv u\  
\hbox{in $\MM$\quad for some sequence $\lambda_k\to 0$}  \right\}.
\]
We aim at showing that the whole family $\{u^\lambda_G\,:\,\lambda>0\}$ converges to a distinguished function  $u^0_G$ as 
$\lambda\to 0^+$, namely that $\omegasetG=\{u^0_G\}$. 
The first step consists in identifying a good candidate $u^0_G$ for the limit of 
the solutions $u^\lambda_G$ of the discounted equations. 
To this aim, we consider the family ${\sol_b}_-(G)$ of bounded subsolutions  $w:\MM\to 
\R$ of the critical equation \eqref{eq G critical} satisfying the following 
condition
\begin{equation}\label{condition u0}
\int_\R w(y)\,d\mu(y)\leqslant \1\big(c(G)-c_f(G)\big) \qquad\text{for every $\mu\in\Mis(G)$}, 
\end{equation}
where $\Mis(G)$ denotes the set of projected Mather measures for $G$ and $\1$ the indicator function of the set $\{0\}$ in the sense of convex analysis, i.e. $\1(t)=0$ if $t=0$ and $\1(t)=+\infty$ otherwise. 
In the case considered in this section, the right-hand side term of \eqref{condition u0} is 0 since $c(G)=c_f(G)$.  

\begin{lemma}\label{label supp projected Mather measures}
The following holds:
\begin{equation}\label{supp projected Mather measures}
\supp(\mu)\subseteq K\qquad\hbox{for all $\mu\in\Mis(G)$.}
\end{equation}
\end{lemma}

\begin{proof}
Let us set $U:=\MM\setminus K$ and choose $r>0$ small enough so that the set $U_r:=\{x\in U\,:\, \D{dist}(x,\partial U)>r\,\}$ is nonempty.  
By arguing as in the proof of Proposition \ref{prop tightness} with $v_G$ in place of $v$, we infer that, for 
every fixed $\eps>0$, there exists a function $v_\eps\in\Lip(\R)\cap\CC^\infty(\R)$ such that 
\begin{eqnarray}\label{ineq strict3}
G(x,v'_\eps(x))\leqslant c(G)+\eps\quad\hbox{in $\MM\setminus U_r$},
\qquad
G(x,v'_\eps(x))\leqslant c(G)-\delta+\eps\quad\hbox{in $U_r$.}
\end{eqnarray}
Pick a Mather measure  $\tilde\mu\in\tilde\Mis (G)$ and let 
$\mu={\pi_1}_\#\tilde\mu$.  
By exploiting Fenchel's inequality, the closed character of $\tilde\mu$ and  \eqref{ineq strict3} we have 
\begin{eqnarray*}
0=\int_{\TM} \Big(L(x,q)+c(G)\Big)\,d\tilde\mu(x,q)
&\geqslant&
\int_{\TM} \Big(v'_\eps(x)q+c(G)-G(x,v'_\eps(x))\Big)\,d\tilde\mu(x,q)\\
&\geqslant&
-\eps+(\delta-\eps)\mu(U_r).
\end{eqnarray*}
By sending first $\eps\to 0^+$ and then $r\to 0^+$, we conclude that $\mu(U)=0$, as it was to be proved. 
\end{proof}

From the previous lemma, we infer that the integral in \eqref{condition u0} is well defined for any 
continuous function $w$. Let us also note that the set ${\sol_b}_-(G)$ is not empty, 
since it contains the bounded and nonpositive subsolution $v_G$ given by 
Theorem \ref{teo strict v_G}.  Furthermore, the following holds. 

\begin{lemma} The family ${\sol_b}_-(G)$ is locally uniformly bounded from 
above in $\R$, i.e.
$$\sup\{w(x)\,:\,u\in {\sol_b}_-(G)\}<+\infty\qquad\hbox{for all $x\in\MM$.}$$
\end{lemma}
\begin{proof} The family of critical subsolutions is equi--Lipschitz. Call 
$\kappa$ a common Lipschitz constant.  
Pick a projected Mather measure $\mu\in\Mis(G)$. In view of  \eqref{supp projected Mather measures}, for every  $w\in {\sol_b}_-(G)$ we have $\min_K w=\int_{\MM} \min_K w\,d\mu\leqslant \int_{\MM} 
w\,d\mu\leqslant 0$.
Hence, $\max_K w\leqslant \max_K w-\min_K 
w\leqslant\kappa\diam K<+\infty$. By using again the fact 
that $w$ is $\kappa$--Lipschitz, 
we conclude that $w(x)\leqslant \kappa\D{dist}(x,K)+\max_K w\leqslant 
\kappa\big(\D{dist}(x,K)+\diam{K}\big)$ for every $w\in{\sol_b}_-(G)$.
\end{proof}
Therefore we can define  $u^0_G:\MM\to \R$ by
\begin{equation}\label{def1 u_0}
u^0_G(x):=\sup_{w\in{\sol_b}_-(G) }w(x),\qquad{x\in\R}.
\end{equation}
As the supremum of a family of viscosity subsolutions to \eqref{eq G critical}, we know that 
$u^0_G$ is itself a critical subsolution of the same equation. 
We will obtain later that $u^0_G$ is a solution, see Theorem \ref{theo main} 
below.
Next, let us show that, in the definition of $u^0_G$, we can remove the constraint 
that critical subsolutions are bounded.

\begin{lemma}\label{lemma reduction}
Let $w\in\Lip(\MM)$ be a subsolution of \eqref{eq G critical}. Then there exists a diverging sequence $(r_n)_n$ in $(0,+\infty)$ and a sequence 
$(w_{n})_n\subset\Lip(\R^{d})$ of bounded 
subsolutions to \eqref{eq G critical} such that $B_{r_n}\subseteq \{w_n=w\}$ for every $n\in\N$. In particular, 
\begin{equation}\label{eq simpler u0}
u^0_G(x)=\sup_{w\in{\sol}_-(G) }w(x)\qquad{x\in\R},
\end{equation}
where ${\sol}_-(G)$ denotes the family of subsolutions to \eqref{eq G critical} satisfying \eqref{condition u0}.
\end{lemma}

\begin{proof}
Let $w\in\Lip(\R)$ be a subsolution to \eqref{eq G critical} and set 
\[
w_n(x)=\min\left\{\max\{w(x),v_G(x)-n\},v_G(x)+n\right\},\qquad x\in\MM,
\]
where $v_G$ is the bounded subsolution to \eqref{eq G critical} provided by Theorem \ref{teo strict v_G}. As a maximum and minimum of critical subsolutions,  $w_n$ is a critical subsolution by Proposition \ref{prop when G convex}. Moreover, $w_n$ is bounded and agrees with $w$ on larger and larger balls as $n\to+\infty$.
In particular, for $n$ large enough, $w_n\equiv w$ on $K$, therefore such a $w_n$ satisfy \eqref{condition u0} whenever 
$w\in\sol_-(G)$ in view of \eqref{supp projected Mather measures}. The remainder of the assertion easily follows from this.  
\end{proof}

Let $u_G $ be the coercive critical solution obtained according to 
Corollary \ref{cor existence u_G}. Since $u_G -\lVert u_G \rVert_{L^\infty(K)}\in {\sol}_-(G)$,  from the previous Lemma we infer that 
\begin{equation}\label{coercive u^0}
u^0_G(x)\geqslant C|x|-1/C\quad\hbox{on $\R$\ \ for some positive constant $C>0$.}
\end{equation}

We now start to study the asymptotic behavior of the discounted value functions 
$u^\lambda_G$ as $\lambda\to 0^+$ and the relation with $u^0_G$.  We begin with the 
following result:

\begin{prop}\label{ineq lim}
Let $\lambda>0$. Then, for every $\mu\in\Mis (G)$, we have 
\[ 
\int_{\MM} u^\lambda_G(x)\, d \mu(x)\leqslant 0.
\]
In particular,  $u\leqslant u^0_G$ on $\MM$ for every $u\in\omegasetG$.
\end{prop}
\begin{proof} By applying Lemma \ref{EncoreUneVariante} with $F(x,p):=\lambda 
u^\lambda_G(x)+G(x,p)-c(G)$, we infer that there exists a sequence $(w_n)_n$ of 
functions in $\D{C}^1({\MM})\cap\Lip(\MM)$ such that 
\hbox{$\|u^\lambda_G-w_n \|_\infty \leqslant 1/n$} and  
\[
 \lambda u^\lambda_G(x)+G(x,w'_n(x)) \leqslant c(G)+1/n\quad \text{ for every $x\in 
{\MM}$.}
\]
By  the Fenchel inequality
\[
L_G(x,q)+G(x,w'_n(x))\geqslant w'_n(x)\,q\quad \text{ for every $(x,q)\in \TM$},
\]
Combining these two inequalities we infer
\begin{equation}\label{FENDISC}
 \lambda u^\lambda_G(x)+ w'_n\,q \leqslant L_G(x,q)+c(G)+\frac1n\quad\text{ for 
every $(x,q)\in \TM$}.
\end{equation}
Pick $\tilde\mu\in\tilde\Mis (G)$, and set 
$\mu={\pi_1}_\#\tilde\mu\in\Mis (G)$. Since $\tilde\mu$ is closed and minimizing, we have 
$\int_{\TM}  w'_n(x)\,q\, d\tilde\mu(x,q)=0$, and 
$\int_{\TM}L_G(x,q)\,d\tilde\mu(x,q)=-c(G)$.
Therefore, if we integrate \eqref{FENDISC}, we obtain
\[
\lambda \int_\MM u^\lambda_G(x)\,d\mu(x)\leqslant \frac1n.
\]
Since $\lambda>0$, letting $n\to\infty$ we obtain $ \int_{\MM} 
u^\lambda_G(x)\,d\mu(x)\leqslant 0$.
If $u$ is the uniform limit of $\left(u^G_{\lambda_n}\right)_n$ for some 
$\lambda_n\to 0$, we know that it is a solution of 
the critical equation \eqref{eq G critical}. Moreover, it also has to satisfy $ 
\int_\MM u(x)\,d\mu(x)\leqslant 0$
 for every projected Mather measure $\mu$.  Therefore  $u\in {\sol}_-(G)$ 
and $u\leqslant u^0_G$. 
 \end{proof}

The next (and final) step consists in showing that $u\geqslant u^0_G$ in $\MM$ whenever $u\in\omegasetG$. 
For this, we will exploit the representation formula
\eqref{representation formula discounted} for $u^\lambda_G$  and Proposition \ref{prop calibrating main}. 

Let us fix $x\in \MM$. For every $\lambda>0$, we  choose 
$\gamma^{\lambda}_x:(-\infty,0]\to \MM$ with $\gamma_x^\lambda(0)=x$ to be an optimal curve for 
$u^\lambda_G$, cf. 
Proposition \ref{prop calibrating main}, and we define a measure 
$\tilde \mu_x^\lambda=\tilde\mu_{\gamma^\lambda_x}^\lambda$ on $\TM$ via 
\eqref{def discounted measure}. The following holds:

\begin{prop}\label{prop discounted measures} 
Let $x\in\MM$ and $(\lambda_n)_n$ be an infinitesimal sequence. Then the set of measures 
$\left(\tilde \mu_x^{\lambda_n}\right)_n$ defined above is a tight family of probability measures,
whose supports are all contained in $\MM\times\overline B_{\tilde\kappa}$ for some 
$\tilde\kappa>0$. In 
particular, they are relatively compact in the space of 
probability measures on $\TM$ with respect to the narrow convergence. Furthermore, 
if $\left(\tilde \mu_x^{\lambda_{n}}\right)_n$ is narrowly converging to 
$\tilde\mu$, then $\tilde\mu$ is a (closed) Mather
measure.
\end{prop}

\begin{proof} 
Call ${\tilde\kappa}$ a common Lipschitz constant for the family of curves 
\hbox{$\{\gamma_x^\lambda\,:\,\lambda>0\}$,} according to Proposition \ref{prop 
calibrating main}. 
Then the measures $\tilde \mu_x^\lambda$ are all probability measures and 
all have support contained in the compact set $\MM\times\overline B_{\tilde\kappa}$, 
as it can be easily checked by their 
definition. In order to prove that the set of probability measures $\left(\tilde \mu_x^{\lambda_n}\right)_n$
is tight, it is enough to prove that 
the set of measures $\left( \mu_x^{\lambda_n}\right)_n$ is tight, where each measure $\mu_x^\lambda:={\pi_1}_\#\tilde\mu^\lambda_x$ 
is defined as the push--forward of the measure $\tilde\mu^\lambda_x$ via the projection map 
${\pi_1}:\TM\ni(x,p)\mapsto x\in\MM$. Moreover, since any finite family of 
probability measure on a Polish space is tight, it suffices to show the 
following:\medskip\\
{\bf\underline{Claim:}} there exists a constant $C>0$ such that \ \ 
$\mu^{\lambda}_x(\R\setminus K)<C\lambda$\quad for every $\lambda>0$.\medskip\\
But this follows as a direct application of Proposition \ref{prop tightness} with $v:=v_G$ and $U:=\MM\setminus K$. By Prohorov's Theorem, see for instance \cite[Theorem 5.1]{Bill99}, we infer that the set of probability 
measures $\left(\tilde \mu_x^{\lambda_n}\right)_n$ is relatively compact in $\parts(\TM)$ with respect to 
narrow convergence. 

Let now assume that  $\left(\tilde 
\mu_x^{\lambda_{n}}\right)_n$ is narrowly converging to $\tilde\mu$ for some 
$\lambda_n\to 0$. Arguing as in the proof of Proposition \ref{prop2 limit measures}-(i), we get that 
$\tilde\mu$ is a closed probability 
measure. It remains to show  that 
$\int_{\TM}L_G(x,q)\,d\tilde\mu(x,q)=-c(G)$. We have
\begin{multline*}
\int_{\TM} \big(L_G(x,q)+c(G)\big)\,d\tilde\mu(x,q)=\lim_{n\to\infty}\int_{\TM} \big(L_G(x,q)+c(G)\big)\,
d\tilde\mu_x^{\lambda_{n}}(x,q)\\
=\lim_{n\to\infty}\int_{-\infty}^0  (\e^{\lambda_n s})'
\big(L_G\big(\gamma_x^{\lambda_n}(s),\dot\gamma_x^{\lambda_n}(s)\big)+c(G)\big)\, d s
=\lim_{n\to\infty}\lambda_nu_G^{\lambda_n}(x)=0,
\end{multline*}
where the last equality follows from the fact that $\lambda u^\lambda_G(x) \to 0$ in view of \eqref{bounds discounted solutions}.
\end{proof}
The following lemma will be crucial for the proof of our main result, see 
Theorem \ref{theo main} below.
\begin{lemma}\label{ineq prim}
Let $w$ be any bounded critical subsolution. For every $\lambda>0$ and $x\in \MM$ 
\begin{equation}\label{eq pre-final}
u^\lambda_G(x)\geqslant w(x)-\int_{\TM} w(y)\, d \tilde{\mu}^{\lambda}_x (y,q).
\end{equation}
\end{lemma}

\begin{proof}
It suffices to remark that the 
measure ${\tilde\mu}^{\lambda}_x$ agrees with the measure ${\tilde\mu}^{\lambda}_{x,1}$ defined 
via \eqref{def2 discounted measure1} where we have taken $r=+\infty$. The assertion follows 
from Proposition \ref{prop step2}.
\end{proof}

We are now ready to prove our main theorem:

\begin{teorema}\label{theo main}
The functions $u^\lambda_G$ uniformly converge to the function $u^0_G$ given by \eqref{def1 u_0} locally on $\MM$ as $\lambda\to 
0^+$.
In particular, as an accumulation point of $u^\lambda_G$ as $\lambda\to 0^+$, the 
function $u^0_G$ is a viscosity solution of \eqref{eq G critical}.
\end{teorema}
\begin{proof}
By Theorem \ref{teo discounted sol} and Proposition \ref{prop application comparison}, we know that the functions 
$u^\lambda_G$ are equi--Lipschitz and locally qui--bounded on $\MM$, 
hence it is enough, by the Ascoli--Arzel\`a theorem, to prove that any 
converging subsequence has  $u^0_G$ as limit.

Let $\lambda_n\to 0$ be such that $u^{\lambda_n}_G$ locally uniformly converge to some 
$u\in\CC(\MM)$. 
We have seen in Proposition \ref{ineq lim} that 
$$ 
u(x)\leqslant u^0_G(x)\qquad\hbox{for every  $x\in \MM$}.
$$
To prove the opposite inequality, let us fix $x\in \MM$.  Let $w$ be a bounded critical 
subsolution.
By Lemma \ref{ineq prim}, we have
\begin{equation*}
u^{\lambda_n}_G(x)\geqslant w(x)-\int_{\TM} w(y)\, d \tilde{\mu}^{\lambda_n}_x (y,q).
\end{equation*}
By Proposition \ref{prop discounted measures}, extracting a further 
subsequence, we can assume that
$\tilde{\mu}^{\lambda_n}_x$ converges narrowly to a Mather measure $\tilde \mu$
whose projection on $\MM$ is denoted by $\mu$. Passing to the limit in the last 
inequality,
we get  
\[
u(x)
\geqslant 
w(x)-\int_{\MM} w(y)\, d {\mu}(y).
\]
If we furthermore assume that 
$w\in {\sol_b}_-$, the set of bounded subsolutions satisfying \eqref{condition u0}, 
we obtain
$\int_{\MM} w(y)\, d {\mu}(y)\leqslant 0$, hence $u\geqslant w$.
We conclude that $u\geqslant u^0_G=\sup\limits_{w\in{\sol_b}_-}w$ in view of Lemma \ref{lemma reduction}.\end{proof}

\section{The case $c(G)=c(H)>c_f(H)$}\label{sez case II}

In this section we shall prove the asymptotic convergence in the case $c(G)=c(H)>c_f(H)$. Throughout this section,  we shall denote by $c$ this common critical constant for notational simplicity. We start with 
some remarks on the critical equations associated with $H$ and $G$.

\subsection{Critical equations}\label{sez critical eq II}
Let us consider the critical equations 
\begin{eqnarray}
 G(x,u')=c\qquad\hbox{in $\R$,} \label{eq2 G critical}\\
 H(x,u')=c\qquad\hbox{in $\R$.} \label{eq2 H critical}
\end{eqnarray}
We introduce a piece of notation first. Let us set $\overline y_V:=\max\big(\supp(V)\big)$, $\underline y_V:=\min\big(\supp(V)\big)$, and, for every $x\in\R$, 
\begin{align*}
Z_H(x)&:=\{p\in\R\,:\,H(x,p)\leqslant c\,\}
 =
 [p^{-}_H(x),p^{+}_H(x)],\\
Z_G(x)&:=\{p\in\R\,:\,G(x,p)\leqslant c\,\}
 =
 [p^{-}_G(x),p^{+}_G(x)].
\end{align*}

From the fact that $c>c_f(H)$ we infer $\rho:=\min_{x\in\R} \big(p_H^+(x)-p_H^-(x)\big)>0$. 
Furthermore, one of the following circumstances occurs, cf. Theorem \ref{teo Sez2 effective ham} (with $\theta=0$)\footnote{We recall that $c(H)=\overline H(0)$, cf. 
Section \ref{sez critical values}.}:
\begin{itemize}
 \item[\bf{(A)}]\quad $\int_0^1 p^+_H(x)\, dx=0>-\rho>\int_0^1 p^-_H(x)\, dx$;\smallskip\\
 \item[\bf{(B)}]\quad $\int_0^1 p^+_H(x)\, dx>\rho>0=\int_0^1 p^-_H(x)\, dx$.\\
\end{itemize}

%

The following holds:

\begin{teorema}\label{teo2 bounded sol}
There exists a bounded solution $u_{G}$ to \eqref{eq2 G critical}. 
\end{teorema}

\begin{proof}
Let us assume we are in case (A). Set $u_G(x):=\int_0^x p_G^+(z)\,dz$ for all $x\in\MM$. It is clear that $u_G$ is a Lipschitz solution to \eqref{eq2 G critical}. To see that it is bounded, simply notice that
\[
u_G(\underline y_V-n)=u_G(y),
\qquad
u_G(\overline y_V+n)=u_G(y)
\qquad
\hbox{for all $n\in\N$,}
\]
since $p^+_G=p^+_H$ in $\MM\setminus\supp\big(V\big)$ and the mean of the 1-periodic function 
$p^+_H$\  is 0. The proof in case (B) is analogous.
\end{proof}

Furthermore, we have the following additional information:

\begin{prop}\label{prop2 bounded G sol}
Let $v$ be a bounded solution to \eqref{eq2 G critical}. Then:
\begin{itemize}
 \item[(i)] in case (A), 
\begin{equation}\label{claim2 bounded G sol}
v'(x)=p_H^+(x)\qquad\hbox{for all $x\geqslant \overline y_V:=\max\big(\supp(V)\big)$.}
\end{equation}
\item[(ii)] in case (B), 
\begin{equation}\label{claim2B bounded G sol}
v'(x)=p_H^-(x)\qquad\hbox{for all $x\leqslant \underline y_V:=\min\big(\supp(V)\big)$.}
\end{equation}
\end{itemize}
\end{prop}

\begin{proof}
Let us prove (i). Let us assume by contradiction that \eqref{claim2 bounded G sol} does not hold. 
Being $v$ a viscosity solution to \eqref{eq2 G critical}, its derivative $v'$ can 
jump only downwards. Since 
\[
 p_G^+(x)-p_G^-(x)
 =
  p_H^+(x)-p_H^-(x)
  \geqslant
  \rho>0
  \qquad\hbox{for all $x\in\MM\setminus\supp\big(V\big)$,}
\]
there exists a point $y\geqslant \overline y_V$ such that 
$v'(x)=p_G^-(x)$ for all $x>y$.  Hence, for all $n\in\N$ we have
\begin{eqnarray*}
 v(y+n)
 &=&
 v(y)+\int_{y}^{y+n} p^-_G(z)dz
 =
 v(y)+\int_{y}^{y+n} p^-_H(z)dz\\
 &=&
 v(y)-n\int_0^1 \big( p_H^+(z)-p_H^-(z)\big) dz
 \leqslant
 v(y)-n\rho,
\end{eqnarray*}
in contrast with  the fact that $v$ is bounded. The proof of item (ii) is analogous. 
\end{proof}

\begin{oss}\label{oss2 bounded G sol}
Proposition \ref{prop2 bounded G sol} holds in particular in the case $G:=H$ 
(i.e. when $V\equiv 0$) with equality in \eqref{claim2 bounded G sol} 
and in \eqref{claim2B bounded G sol} holding for every $x\in\MM$, by periodicity of the solution $v$.
\end{oss}

\subsection{Asymptotic convergence}\label{sez asymptotics II}
Let us denote by $u^\lambda_G$, $u^\lambda_H$ the solutions of the discounted equations
\begin{eqnarray}
 \lambda u+ G(x,u')=c\qquad\hbox{in $\R$,} \label{eq2 G discounted}\label{eq discount G}\\
 \lambda u + H(x,u')=c\qquad\hbox{in $\R$.} \label{eq2 H discounted}\label{eq discount H}
\end{eqnarray}

We want to prove that the solutions $u^\lambda_G$ converge, as $\lambda\to 0^+$, to a specific solution of $u^0_G$ the critical equation \eqref{eq2 G critical}. 
In view of \cite{DFIZ1}, we know that this is true for the solutions of \eqref{eq2 H discounted}, i.e. $u^\lambda_H\ucv u^0_H$ in $\MM$ where $u^0_H$ is a solution 
to the critical equation \eqref{eq2 H discounted}, and this convergence 
is actually uniform in $\MM$ since all these functions are $\Z$--periodic. 
Furthermore, $u^0_H$ is identified as the maximal (sub-)solution to 
\eqref{eq2 H critical} such that 
\begin{equation}\label{eq2 u^0_H}
\int_{\TTorus} v(y)\,d\mu(y)\leqslant 0\qquad\hbox{for all $\mu\in\Mis(H)$,}
\end{equation}
where $\Mis(H)$ denotes the set of projected Mather measures for $H$.

Let $u_G$ be the bounded solution of \eqref{eq2 G critical} given by 
Theorem \ref{teo2 bounded sol}. Since $u_G^-:=u_G-\|u_G\|_\infty$ and 
$u_G^+:=u_G+\|u_G\|_\infty$ are, respectively, a bounded negative and positive solution 
of \eqref{eq2 G critical}, we derive from Proposition
\ref{prop application comparison} that 
the functions $\{u^\lambda_G\,:\,\lambda>0\}$ are equi--bounded and equi--Lipschitz on $\MM$, and satisfy 
\begin{equation}\label{eq2 equibounded discounted sol}
-2\|u_G\|_\infty
\leqslant 
u^-_G(x)
\leqslant 
u^\lambda_G(x)
\leqslant 
u^+_G(x)
\leqslant 
2\|u_G\|_\infty
\qquad
\hbox{for all $x\in\MM$.}
\end{equation} 
Let us denote by 
\[
\omegasetG:=\left\{u\in\Lip(\MM)\,: \, u^{\lambda_k}_G\ucv u\  
\hbox{in $\MM$\quad for some sequence $\lambda_k\to 0$}  \right\}.
\]
Of course, our aim is to show that this set reduces to $\{u^0_G\}$. 
We first have to identify a good candidate for $u^0_G$. 
%
To this aim, we start with some preliminary remarks. 
For every $\lambda>0$, let us set 
\begin{align*}
Z_H^\lambda(x)&:=\{p\in\R\,:\,H(x,p)\leqslant c- \lambda u^\lambda_H(x)\,\}
 =
 [p^{\lambda,-}_H(x),p^{\lambda,+}_H(x)]
 \qquad
 \hbox{for all $x\in\MM$,}\\
Z_G^\lambda(x)&:=\{p\in\R\,:\,G(x,p)\leqslant c- \lambda u^\lambda_G(x)\,\}
 =
 [p^{\lambda,-}_G(x),p^{\lambda,+}_G(x)]
 \qquad
 \hbox{for all $x\in\MM$.}
\end{align*}
We shall need the following technical lemma. 

\begin{lemma}\label{lemma continuity modulus}
There exists a modulus of continuity $\omega$ such that, for every $\lambda\in (0,1)$,  
\begin{itemize}
\item[(i)] \quad $\|p^{\lambda,-}_G-p^-_G\|_\infty \leqslant \omega(\lambda),
 \qquad
 \|p^{\lambda,+}_G-p^+_G\|_\infty \leqslant \omega(\lambda)$;\smallskip
\item[(ii)] \quad $\|p^{\lambda,-}_H-p^-_H\|_\infty \leqslant \omega(\lambda),
 \qquad
 \|p^{\lambda,+}_H-p^+_H\|_\infty \leqslant \omega(\lambda)$.
\end{itemize}
\end{lemma}

\begin{proof}
Let $C>0$ be such that
\begin{equation*}
{\|u^\lambda_H\|}_\infty+{\|u^\lambda_G\|}_\infty<C 
\qquad\hbox{for all $\lambda\in (0,1)$.}
\end{equation*}
Within this proof, we shall use the following temporary notation:
\begin{align*}
Z_H(x,a)&:=\{p\in\R\,:\,H(x,p)\leqslant a\,\}
 =
 [p^-_H(x,a),p^{+}_H(x,a)],\\
Z_G(x,a)&:=\{p\in\R\,:\,G(x,p)\leqslant a\,\}
 =
 [p^{-}_G(x,a),p^{+}_G(x,a)].
\end{align*}
Furthermore, we shall denote by $\D{dom}(Z_H)$ (respectively, $\D{dom}(Z_G)$) the set of points 
$(x,a)\in\MM\times\R$ such that $Z_H(x,a)$ (respectively,  $Z_G(x,a)$) is nonempty. 
Notice that 
$$
\D{dom}(Z_H)\setminus \big(\supp(V)\times\R\big) 
=\D{dom}(Z_G)\setminus \big(\supp(V)\times\R\big).
$$
The functions $(x,a)\mapsto p^{\pm}_H(x,a)$ and 
$(x,a)\mapsto p^{\pm}_G(x,a)$ are continuous in $\D{dom}(Z_H)$ and $\D{dom}(Z_G)$, respectively, 
in view of Lemma \ref{lemma properties sigma_a} and Remark \ref{oss properties sigma_a}. 
Furthermore, the functions $p^{\pm}_H(x,a)$ are $1$--periodic in $x$ and 
$p^-_G(x,a)=p^-_H(x,a)$, $p^+_G(x,a)=p^+_H(x,a)$ for every 
$(x,a)\in \D{dom}(Z_G)\setminus \big(\supp(V)\times\R\big)$. 
We derive that there exists a modulus of continuity $\tilde\omega$ such that 
\begin{equation}\label{eq continuity modulus}
|p^-_G(x,a)-p^-_G(x,b)|\leqslant \tilde\omega(|a-b|),
\quad
|p^+_G(x,a)-p^+_G(x,b)|\leqslant \tilde\omega(|a-b|)
\end{equation}
for all $(x,a),(x,b)\in\D{dom}(Z_G)$. 
Now notice that $p^\pm_G(x)=p^\pm_G(x,c)$ and $
p^{\lambda,\pm}_G(x)=p^\pm_G(x,c-\lambda u^\lambda_G(x))$ for all $x\in\MM$. 
Furthermore, it is clear 
that $(x,c)$ and $(x,c-\lambda u^\lambda_G(x))$ both belong to $\D{dom}(Z_G)$ for every $x\in\MM$. 
By applying \eqref{eq continuity modulus} with $a:=c$ and $b:=c+\lambda u^\lambda_G(x)$, we get 
assertion (i) with $\omega(h):=\tilde\omega(Ch)$ for every $h\geqslant 0$. The same argument with $H$ 
in place of $G$ gives assertion (ii). 
\end{proof}

We proceed by showing the following result. 

\begin{prop}\label{prop2 G discounted sol}
There exist $\lambda_0>0$ and $N_0\in\N$ such that for every $\lambda\in (0,\lambda_0)$
the following holds:
\begin{itemize}
 \item[(i)] in case (A), 
\begin{align}
\big(u^\lambda_G\big)'(x)=p_G^{\lambda,+}(x)&
\qquad
\hbox{for all $x\in (-\infty,\underline y_V-N_0)\cup(\overline y_V,+\infty)$}
\label{claim2 G discounted sol}\\
\big(u^\lambda_H\big)'(x)=p_H^{\lambda,+}(x)&
\qquad
\hbox{for all $x\in \MM$};\nonumber
\end{align}
\item[(ii)] in case (B), 
\begin{align*}
\big(u^\lambda_G\big)'(x)=p_G^{\lambda,-}(x)&
\qquad
\hbox{for all $x\in (-\infty,\underline y_V)\cup(\overline y_V+N_0,+\infty)$}\\
\big(u^\lambda_H\big)'(x)=p_H^{\lambda,-}(x)&
\qquad
\hbox{for all $x\in\MM$}.
\end{align*}
\end{itemize}
\end{prop}

\begin{proof}
We only prove (i), being the proof of (ii) similar. 
Let $C>0$ be such that
\begin{equation}\label{eq C bound}
{\|u^\lambda_H\|}_\infty+{\|u^\lambda_G\|}_\infty<C 
\qquad\hbox{for all $\lambda\in (0,1)$.}
\end{equation}
For every fixed $\eps>0$, choose $\lambda(\eps)>0$ small enough so that 
$\lambda C<\eps$\ \ for all $\lambda\in (0,\lambda(\eps))$.
Now we choose $\hat\eps>0$ small enough so that $c-\hat\eps>c_f(H)$ and, thanks to 
Lemma \ref{lemma continuity modulus}, 
\[
 {\|p^{\lambda,-}_G-p^-_G\|}_\infty<\dfrac{\rho}{4},
 \qquad
 {\|p^{\lambda,+}_G-p^+_G\|}_\infty<\dfrac{\rho}{4}
 \qquad\qquad\hbox{for all $\lambda\in (0,\lambda(\hat\eps))$},
\]
where $\rho:=\min_{x\in\R} \big(p_H^+(x)-p_H^-(x)\big)>0$. 
Set $\lambda_0:=\lambda(\hat\eps)$ and fix $\lambda\in (0,\lambda_0)$. Due to the fact that 
$p_G^\pm\equiv p_H^\pm$\  on $\MM\setminus\supp(V)$, we infer 
\begin{equation}\label{eq2 length Z^lambda_G}
 p_G^{\lambda,+}(x)-p_G^{\lambda,-}(x)
 \geqslant
 p_G^{+}(x)-p_G^{-}(x)-\dfrac{\rho}{2}
 \geqslant \dfrac{\rho}{2}
 \qquad
 \hbox{for all $x\in\MM\setminus\supp(V)$}
\end{equation}
and 
\[
 \int_{y}^{y+1} p_G^{\lambda,-}(z) dz
 \leqslant 
 \dfrac\rho4+\int_{y}^{y+1} p_G^{-}(z) dz
 \qquad
 \hbox{for all $y\in\MM$,}
\]
in particular 
\begin{equation}\label{eq2 mean discounted p^-}
\int_y^{y+1} p_G^{\lambda,-}(z)\, dz
\leqslant -\dfrac{3}{4}\rho
\qquad
\hbox{whenever $(y,y+1)\cap\supp(V)=\emptyset$}
\end{equation}
due to (A).  
Let us assume that \eqref{claim2 G discounted sol} does not hold. 
Being $u_G^\lambda$ a viscosity solution of \eqref{eq2 G discounted}, its 
derivative can jump only downwards. In view of \eqref{eq2 length Z^lambda_G}, 
two possibilities may occur: 
\begin{itemize}
 \item[(a)] $\big(u^\lambda_G\big)'(x)=p^{\lambda,-}_G(x)$\quad for all $x>y$\qquad for some 
 $y\geqslant \overline y_V$;\smallskip
 \item[(b)] there exists $y\leqslant\underline y_V$\  such that 
 \[
  \big(u^\lambda_G\big)'(x)=p^{\lambda,+}_G(x)\quad \hbox{in $(-\infty,y)$}
  \qquad
  \hbox{and} 
  \qquad
  \big(u^\lambda_G\big)'(x)=p^{\lambda,-}_G(x)\qquad \hbox{in $(y,\underline y_V)$}.
 \]
\end{itemize}
Let us proceed to show that instance (a) never occurs. Otherwise, we would have 
\[
 u^\lambda_G(y+n)-u^\lambda_G(y)
 =
 \int_y^{y+n} p^{\lambda,-}_G(z)\ dz
 \leqslant
 -\dfrac34n\rho
 \qquad
 \hbox{for all $n\in\N$,}
 \]
in contradiction with the fact that $u^\lambda_G$ is bounded. In case  (b), let us denote 
by $N$ the integer part of $\underline y_V-y$. Then 
\[
 u^\lambda_G(\underline y_V)-u^\lambda_G(\underline y_V-N)
 =
 \int_{\underline y_V-N}^{\underline y_V} p^{\lambda,-}_G(z)\ dz
 \leqslant
 -\dfrac34N\rho,
\]
yielding 
\[
N
\leqslant 
\dfrac{8\|u^\lambda_G\|_\infty}{3\rho}
\leqslant
\dfrac{8 C}{3\rho}.
\]
in view of \eqref{eq C bound}. 
Assertion \eqref{claim2 G discounted sol} follows by setting 
$N_0:=1+[8C/(3\rho)]$ (where the symbol $ [h] $ stands for the integer part of $h$). 
The same argument with $V\equiv 0$ proves the second assertion in (i) with same constant 
$\lambda_0>0$. 
\end{proof}

In the remainder of this section, $\hat \eps$ will denote a positive parameter chosen small enough so that 
$c-\hat\eps>c_f(H)$, and $\lambda_0=\lambda(\hat\eps)$ will denote the positive parameter 
chosen as in the proof of Proposition \ref{prop2 G discounted sol}.

\begin{prop}\label{prop2 optimal curve}
Let $\lambda\in (0,\lambda_0)$,  $x\in\MM$ and 
$\gamma^\lambda_x:(-\infty,0]\to\MM$ be a minimizing curve for $u^\lambda_G(x)$. The 
following holds:
\begin{itemize}
 \item[(i)] in case (A),\quad $\gamma^\lambda_x(t)\leqslant \max\{x,\overline y_V\}$\quad  
 for all $t\leqslant 0$;\smallskip
%
\item[(ii)] in case (B),\quad $\gamma^\lambda_x(t)\geqslant \min\{x,\underline y_V\}$\quad  
 for all $t\leqslant 0$.
\end{itemize}
\end{prop}

\begin{proof}
We only prove (i), being the proof of item (ii) analogous. 
Let us assume by contradiction that there exists $a<0$ such that 
$\gamma_x^\lambda(a)>\max\{x,\overline y_V\}$. Let us set 
\[
b:=\sup\{\beta\in(a,0)\,:\,\gamma_x^\lambda(t)>\max\{x,\overline y_V\}\ \hbox{for all $t\in (a,\beta)$}\,\}.
\]
Since $\gamma_x^\lambda$ is optimal for $u^\lambda_G(x)$, in view of 
Remark \ref{oss calibrating curves} and of Proposition \ref{prop2 G discounted sol} we have
\[
\dot\gamma_x^\lambda(t)
\in 
\partial_p G\big(\gamma_x^\lambda(t),(u^\lambda_G)'(\gamma_x^\lambda(t))\big)
=
\partial_p H\big(\gamma_x^\lambda(t),p_G^{\lambda,+}(\gamma_x^\lambda(t))\big)
\qquad
\hbox{for a.e. $t\in (a,b)$}, 
\]
hence\  $\dot\gamma_x^\lambda(t)>0$ \  for a.e. $t\in (a,b)$ due to the fact that (since we are in case (A))
\begin{eqnarray*}
H\big(\gamma_x^\lambda(t),p_G^{\lambda,+}(\gamma_x^\lambda(t))\big)
=
G\big(\gamma_x^\lambda(t),(u^\lambda_G)'(\gamma_x^\lambda(t))\big)
=
c-\lambda u^\lambda_G(\gamma_x^\lambda(t))
>
c-\hat\eps
>
c_f(H)
\end{eqnarray*}
for every $t\in (a,b)$ and $c_f(H)\geqslant\min_{p} H(\gamma_x^\lambda(t),p)$. 
We derive $\gamma_x^\lambda(b)>\max\{x,\overline y_V\}$. 

This gives the sought contradiction when $b<0$, by maximality of $b$, and also when $b=0$, since $\gamma_x^\lambda(0)=x$.  
%
\end{proof}

The previous argument in the case $V\equiv 0$ gives the following result.

\begin{cor}\label{cor2 optimal curve}
Let $\lambda\in (0,\lambda_0)$,  $x\in\MM$ and 
$\gamma^\lambda_x:(-\infty,0]\to\MM$ be a minimizing curve for $u^\lambda_H(x)$. 
In case (A), $\gamma^\lambda_x(t)<x\ 
\hbox{for all $t<0$.}$ In case (B),  
$\gamma^\lambda_x(t)>x\ \hbox{for all $t<0$.}$
\end{cor}

As a consequence of the information gathered, we obtain the following relation 
between the solutions of the discounted equations \eqref{eq2 G discounted} and 
\eqref{eq2 H discounted}. 

\begin{cor}\label{cor2 u^lambda_G=u^lambda_H}
Let $\lambda\in(0,\lambda_0)$. The following holds:
\begin{itemize}
 \item[(i)] in case (A),\quad $u^\lambda_G(x)\leqslant u^\lambda_H(x)$
\ for all $x\in(-\infty,\underline y_V)$;\smallskip
\item[(ii)] in case (B),\quad $u^\lambda_G(x)\leqslant u^\lambda_H(x)$
\ for all $x\in(\overline y_V,+\infty)$. 
\end{itemize}
\end{cor}

\begin{proof}
Let us prove (i). Let $x<\underline y_V$ and $\gamma_x^\lambda:(-\infty,0]\to\MM$ be a minimizing curve for 
$u^\lambda_H(x)$.  According to Corollary \ref{cor2 optimal curve}, the support of $\gamma_x^\lambda$ does not intersect $\supp(V)$. We derive 
\[
u^\lambda_H(x)
=
\int_{-\infty}^0 e^{\lambda s} \big(L_H(\gamma_x^\lambda,\dot\gamma_x^\lambda)+c\big)\ ds
=
\int_{-\infty}^0 e^{\lambda s} \big(L_G(\gamma_x^\lambda,\dot\gamma_x^\lambda)+c\big)\ ds
\geqslant
u^\lambda_G(x).
\]
\end{proof}

It is time to make our guess about the expected limit $u^0_G$ 
of the solutions $u^\lambda_G$ of the discounted equations \eqref{eq2 G discounted} 
as $\lambda\to 0^+$. We introduce a piece of notation first: let us denote by 
${\sol_b}_-(G)$ the family of bounded 
subsolutions  $v:\MM\to 
\R$ of the critical equation \eqref{eq2 G critical} satisfying 
\begin{equation}\label{condition2 u0}
\int_\R v(y)\,d\mu(y)\leqslant \1\big(c(G)-c_f(G)\big) \qquad\text{for every $\mu\in\Mis(G)$}, 
\end{equation}
where we have denoted by $\Mis(G)$ the set of projected Mather measures for $G$ and by $\1$ the indicator function of the set $\{0\}$ in the sense of convex analysis, i.e. $\1(t)=0$ if $t=0$ and $\1(t)=+\infty$ otherwise.  
Note that constraint \eqref{condition2 u0} is empty whenever $c(G)>c_f(G)$.
Motivated by the information gathered so far, we make the following guess:
\begin{itemize}
 \item[(a)] in case (A), for every $x\in\MM$ we set 
 \begin{equation}\label{def2A u_0}
 u^0_G(x):=\sup\{v(x)\,:\,v\in{\sol_b}_-(G),\ v\leqslant u^0_H\quad 
 \hbox{in $(-\infty,\underline y_V)$\ }\};
  \end{equation}
 \item[(b)] in case (B), for every $x\in\MM$ we set 
 \begin{equation}\label{def2B u_0}
 u^0_G(x):=\sup\{v(x)\,:\,v\in{\sol_b}_-(G),\ v\leqslant u^0_H\quad 
 \hbox{in $(\overline y_V,+\infty)$\ }\}.
 \end{equation}
\end{itemize}

Our next proposition shows in particular that the sets appearing at the right-hand 
side of formulae \eqref{def2A u_0} and \eqref{def2B u_0} are nonempty.

\begin{prop}\label{prop step1}
Let $u\in \omegasetG$. Then 
\begin{equation}\label{claim step1}
\int_{\MM} u\,d\mu\leqslant 0\quad\hbox{for all $\mu\in\Mis(G)$}.
\end{equation}
Furthermore, $u\leqslant u^0_H$  in $(-\infty, \underline y_V)$ in case (A), while $u\leqslant u^0_H$ in 
$(\overline y_V,+\infty)$ in case (B). In particular,  
$u\leqslant u^0_G$ in $\MM$. 
\end{prop}

\begin{proof}
By arguing as in the proof of Proposition \ref{ineq lim} we infer that $u^\lambda_G$ 
satisfies \eqref{claim step1} for every $\lambda>0$, hence the same holds for $u\in\omegasetG$. 
Furthermore, in view of Corollary \ref{cor2 u^lambda_G=u^lambda_H}, we also know that $u\leqslant u^0_H$  
in $(-\infty, \underline y_V)$ if we are in case (A), and $u\leqslant u^0_H$   in 
$(\overline y_V,+\infty)$, if we are in case (B). In either case, 
$u\leqslant u^0_G$ in $\MM$.
\end{proof}

Next, we show that $u^0_G$ is well-defined. 

\begin{prop}
The function $u^0_G$ is a bounded viscosity subsolution of equation \eqref{eq2 G critical}.
\end{prop}

\begin{proof}
As a supremum of a nonempty and equi-Lipschitz family of subsolutions 
to \eqref{eq2 G critical}, the function $u^0_G$ is a subsolution to the 
same equation provided it is finite-valued.  
Let us then prove that $u^0_G$ is bounded on $\MM$. 
In view of Proposition \ref{prop step1}, it is enough to show that $u^0_G$ is bounded from 
above. Let us assume for definiteness that we are in case (A). Let us set 
\[
\tilde u(x):=u^0_H(\underline y_V)+\int_{\underline y_V}^x p_G^+(z)\,dz
\qquad
\hbox{for all $x\in\MM$.}
\]
It is easily seen that \ $\tilde u'(x)=p_G^+(x)$ for every $x\in\MM$, in particular  
$\tilde u$ is a (classical) solution 
to \eqref{eq2 G critical} such that $\tilde u\equiv u^0_H$ on $(-\infty,\underline y_V]$, cf.  
Remark \ref{oss2 bounded G sol}. Furthermore,  
\[
\tilde u(x)-\tilde u(\overline y_V)
=
\int_{\overline y_V}^x p^+_G(z)\,dz
=
\int_{\overline y_V}^x p^+_H(z)\,dz
=
u^0_H(x)-u^0_H(\overline y_V)
\qquad
\hbox{for all $x\geqslant \overline y_V$},
\]
thus showing that $\tilde u$ is bounded. Let $v$ be a bounded subsolution to \eqref{eq2 G critical} such 
that  $v\leqslant u^0_H$ \ in $(-\infty,\underline y_V)$. For every 
$x>\underline y_V$ we have 
\[
v(x)=v(\underline y_V)+\int_{\underline y_V}^x v'(z)\,dz
\leqslant
u^0_H(\underline y_V)+\int_{\underline y_V}^x p_G^+(z)\,dz
=
\tilde u(x).
\]
This shows that \ $u^0_G(x)\leqslant \tilde u(x) \leqslant \|\tilde u\|_\infty$ for every $x\in\MM$. The proof in 
case (B) is analogous.
\end{proof}

We are now in position to prove the asymptotic convergence. 

\begin{teorema}\label{theo main2}
The functions $u^\lambda_G$ uniformly converge to the function  
$u^0_G$ given either by \eqref{def2A u_0} or by \eqref{def2B u_0} 
locally on $\MM$ as $\lambda\to 
0^+$.
In particular, as an accumulation point of $u^\lambda_G$ as $\lambda\to 0^+$, the 
function $u^0_G$ is a viscosity solution of \eqref{eq2 G critical}.
\end{teorema}

\begin{proof}
By Proposition \ref{prop application comparison} we know that the functions 
$u^\lambda_G$ are equi--Lipschitz and equi--bounded, 
hence it is enough, by the Ascoli--Arzel\`a theorem, to prove that any 
converging subsequence has  $u^0_G$ as limit.

Let $\lambda_n\to 0$ be such that $u^{\lambda_n}_G$ locally uniformly converge to some 
$u\in\CC(\MM)$. 
We have seen in Proposition \ref{prop step1} that 
$$ 
u(x)\leqslant u^0_G(x)\qquad\hbox{for every  $x\in \MM$}.
$$
Let us prove the opposite inequality. Let us assume, for definiteness, that we are in 
case (A). Fix $x\in \MM$ and pick $v\in{\sol_b}_-(G)$ such that 
$v\leqslant u^0_H$ in $(-\infty,\underline y_V)$. 
For every $\lambda>0$, let $\gamma^{\lambda}_x:(-\infty,0]\to \MM$ with $\gamma_x^\lambda(0)=x$ to be a 
monotone minimizer for the variational 
formula related to $u^\lambda_G(x)$, see  
Proposition \ref{prop calibrating main} and Theorem \ref{teo optimal curves}. 
Choose $r>0$ big enough so that 
$\{x\}\cup [\underline y_V,\overline y_V]\subseteq B_r$ and let  
$\tilde\mu_{x,1}^\lambda,\,\tilde\mu_{x,2}^\lambda$ be the probability measures on $\TM$ 
defined via \eqref{def2 discounted measure1} and  \eqref{def2 discounted measure2}, respectively. 
When $T_x^\lambda<+\infty$, we infer from Proposition \ref{prop2 optimal curve} that 
$\gammaxlambda\big((-\infty,-T^\lambda_x)\big)\subseteq (-\infty,-r)$, consequently 
$\supp(\tilde\mu^\lambda_x)\subset (-\infty,-r)\times\R$. 
%
%
In view of Proposition \ref{prop step2} and of the fact that 
$v\leqslant u^0_H$ in $(-\infty,-r)$, we get, for every $\lambda>0$,  
\[
 u^\lambda_{G}(x)\geqslant v(x)
 -
 \left(
 \theta^\lambda_x\int_{\TM} v(y)\, d{\tilde\mu}^\lambda_{x,1}(y,q)
 +
 (1-\theta^\lambda_x)\int_{\TM} u_H^0(y)\, d{\tilde\mu}^\lambda_{x,2}(y,q)
 \right).
\]
Now set $\lambda:=\lambda_n$ and send $n\to +\infty$ in the inequality above. In view of 
Proposition \ref{prop2 precompactness}, there exist  
measures $\tilde\mu_{1,x}\in\parts(\TM)$ and $\tilde\mu_{2,x}\in\parts(\TTorus)$ and 
$\theta\in [0,1]$ such that
\begin{equation}\label{eq2a quasi claim}
  u(x)\geqslant v(x)-\theta\int_{\TM} v(y)\, d{\tilde\mu}_{x,1}(y,q)
 -
 (1-\theta)\int_{\TTorus} u_H^0(y)\, d{\tilde\mu}_{x,2}(y,q).
\end{equation}
Furthermore, in view of Propositions \ref{prop2 limit measures}, 
\[
\int_{\TM} v(y)\, d{\tilde\mu}_{x,1}(y,q)\leqslant 0\quad\hbox{if $\theta\not=0$},
\qquad\quad 
\int_{\TTorus} u_H^0(y)\, d{\tilde\mu}_{x,2}(y,q)\leqslant 0\quad\hbox{if $\theta\not=1$.}
\]
By exploiting this information in \eqref{eq2a quasi claim}, we get $u(x)\geqslant v(x)$. 
Hence 
\[
 u(x)
 \geqslant
 \sup\{v(x)\,:\,v\in{\sol_b}_-(G),\ v\leqslant u^0_H\quad 
 \hbox{in $(-\infty,\underline y_V)$\,}\}
  =
 u^0_G(x)
\]
for all $x\in\MM$, as it was to be shown. The proof for the case (B) is analogous. 
\end{proof}

\section{The case $c(G)=c(H)=c_f(H)$}\label{sez case III}

In this section we shall prove the asymptotic convergence when $c(G)=c(H)=c_f(H)$. 
Since $c(G)\geqslant c_f(G)\geqslant c_f(H)$, in this case we furthermore have 
$c(G)=c(H)=c_f(G)=c_f(H)$.  We will denote by $c$ this common constant, for notational simplicity.  
We start with some remarks on the critical equations associated with $H$ and $G$.

\subsection{Critical equations}\label{sez critical eq III}
Let us therefore consider the critical equations 
\begin{eqnarray}
 G(x,u')=c\qquad\hbox{in $\R$,} \label{eq3 G critical}\\
 H(x,u')=c\qquad\hbox{in $\R$.} \label{eq3 H critical}
\end{eqnarray}

As in Section \ref{sez case II}, we set $\overline y_V:=\max\big(\supp(V)\big)$, $\underline y_V:=\min\big(\supp(V)\big)$, and, for every $x\in\R$, 
\begin{align*}
Z_H(x)&:=\{p\in\R\,:\,H(x,p)\leqslant c\,\}
 =
 [p^{-}_H(x),p^{+}_H(x)],\\
Z_G(x)&:=\{p\in\R\,:\,G(x,p)\leqslant c\,\}
 =
 [p^{-}_G(x),p^{+}_G(x)].
\end{align*}

In view of Theorem \ref{teo Sez2 effective ham} (with $\theta=0$)\footnote{We recall that $c(H)=\overline H(0)$, cf. 
Section \ref{sez critical values}.}, the following holds:
\begin{equation}
P_H^-:=
\int_0^1 p_H^-(z)\ dz 
\leqslant
0
\leqslant
\int_0^1 p_H^+(z)\ dz
=:P_H^+.
\end{equation}
Let us denote by $S_H$ and $S_G$ the critical semi-distance associated with  
$H$ and $G$ via \eqref{eq S} with $a=c$, respectively. A direct computation 
shows that 
\begin{align*}
S_H(y,x)=\int_y^x p^+_H(z)\ dz, 
&\qquad S_G(y,x)=\int_y^x p^+_G(z)\ dz 
&\quad\hbox{if $y\leqslant x$}\\
S_H(y,x)=\int_x^y -p^-_H(z)\ dz, 
&\quad S_G(y,x)=\int_x^y -p^-_G(z)\ dz &\hbox{if $y> x$}
\end{align*}
As a consequence, we derive the following result:
\begin{prop}\label{prop3 properties S}
For every $y\in\MM$ and $n\in\N$, we have 
\[
 S_H(y,y+n)=nP_H^+,
 \qquad
 S_H(y,y-n)=n(-P_H^-).
\]
Furthermore, there exists positive constants $C$ and $\kappa$, such that \quad 
\[
-C
\leqslant 
S_G(y,x)
\leqslant 
\kappa|x-y|\qquad\hbox{for all $x,y\in\MM$.}
\]
\end{prop}

\begin{proof}
The first assertion is apparent from the formulae for $S_H$ provided above and 
the periodicity of $p^\pm_H$. 
Let us consider the second assertion. The lower bound 
for $S_G$ follows via the same argument used in the proof of 
Proposition \ref{prop S_c coercive}, with the only difference that 
here $\delta=0$ since $c(G)=c(H)$. The upper bound is a general fact already remarked in Section \ref{sez HJ equations}.
\end{proof}

We record here for later use the following lemma. 

\begin{lemma}\label{lemma3 S_H min}
For every $y\in\R$, we have
\[
 \inf\{S_H(\xi,x),\:\,\xi\in y+\Z\,\}
 =
 \min\{S_H(\xi,x),\:\,\xi\in \big(y+\Z\big)\cap[x-1,x+1]\,\}
 \qquad
 \hbox{for all $x\in\MM$.}
\]
\end{lemma}

\begin{proof}
It is clearly enough to prove the following assertion:\medskip\\
{\bf\underline{Claim:}} let $\xi\in y+\Z$. If $|\xi-x|>1$, there exists $\eta\in y+\Z$ such that 
\[
 |\eta-x|\leqslant|\xi-x|-1
 \qquad
 \hbox{and}
 \qquad
 S_H(\eta,x)\leqslant S_H(\xi,x).
\]
Let us divide the proof in two cases. 

If $x-\xi>1$, i.e. $\xi<x-1$, we have 
\[
 S_H(\xi,x)
 =
 S_H(\xi,\xi+1)+S_H(\xi+1,x)
 =
 P_H^+ +S_H(\xi+1,x)
 \geqslant
 S_H(\xi+1,x),
\]
and the claim holds true with $\eta:=\xi+1$. 

If $\xi-x>1$, i.e. $\xi>x+ 1$, we have 
\[
 S_H(\xi,x)
 =
 S_H(\xi,\xi-1)+S_H(\xi-1,x)
 =
 (-P_H^-) +S_H(\xi-1,x)
 \geqslant
 S_H(\xi-1,x),
\]
and the claim holds true with $\eta:=\xi-1$. 
\end{proof}

Let us denote by $\E(H)$ and $\E(G)$ the sets of equilibria associated with  $H$ and $G$, 
respectively. Since $G=H$ in $\MM\setminus\supp(V)$, we have 
\begin{equation}\label{eq3 equilibria}
\E(H)\setminus\supp(V)
=
\E(G)\setminus\supp(V).
\end{equation}

The following holds:

\begin{teorema}\label{teo3 bounded G sol}
There exists a bounded solution $u_G$ to \eqref{eq3 G critical}. 
\end{teorema}

\begin{proof}
Let us set\ $u_G(x):=\inf_{y\in\E(G)} S_G(y,x)$. 
From \eqref{eq3 equilibria}, the fact that $\E(H)$ is $\Z$-periodic and 
of Proposition \ref{prop3 properties S}, we derive that $u_G$ is a well defined 
bounded, Lipschitz function. 
Furthermore, as an infimum of viscosity solutions of \eqref{eq3 G critical}, we get that $u_G$ is a viscosity solution as well in view of Proposition \ref{prop when G convex}. 
\end{proof}

\subsection{Asymptotic convergence} \label{sez asymptotics III}
Let us denote by $u^\lambda_G$, $u^\lambda_H$ the solutions of the discounted equations
\begin{eqnarray}
 \lambda u+ G(x,u')=c\qquad\hbox{in $\R$,} \label{eq3 G discounted}\\
 \lambda u + H(x,u')=c\qquad\hbox{in $\R$.} \label{eq3 H discounted}
\end{eqnarray}

We want to prove that the solutions $u^\lambda_G$ converge, as $\lambda\to 0^+$, to a specific solution of $u^0_G$ the critical equation \eqref{eq3 G critical}. 
In view of \cite{DFIZ1}, we know that this is true for the solutions of \eqref{eq3 H discounted}, i.e. $u^\lambda_H\ucv u^0_H$ in $\MM$ where $u^0_H$ is a solution 
to the critical equation \eqref{eq3 H discounted}, and this convergence 
is actually uniform in $\MM$ since all these functions are $\Z$--periodic. Furthermore, $u^0_H$ is identified as the maximal (sub-)solution 
$v$ to 
\eqref{eq3 H critical} such that 
\begin{equation}\label{eq3 u^0_H}
\int_{\TTorus} v(y)\,d\mu(y)\leqslant 0\qquad\hbox{for all $\mu\in\Mis(H)$,}
\end{equation}
where $\Mis(H)$ denotes the set of projected Mather measures for $H$.

The asymptotic convergence of the solutions $u^{\lambda}_G$ will be established under the following assumption:
\begin{equation}\label{condition u_0}
u^0_H(x)=\sup\{v(x)\,:\,\hbox{$v$ is a $\Z$-periodic  subsolution of \eqref{eq3 H critical} with $v\leqslant 0$ on $\E(H)$}\,\}. \tag{U}
\end{equation}

This is motivated by the following fact. 

\begin{prop}\label{prop3 u_0}
Condition \eqref{condition u_0} is fulfilled in either one of the following cases:
\begin{itemize}
\item[(i)] when $\int_0^1 p_H^-(x) dx < 0< \int_0^1 p_H^+(x) dx$;\smallskip
\item[(ii)] when  $\big(\int_0^1 p_H^-(x) dx\big)\big(\int_0^1 p_H^+(x) dx \big)=0$  and $H$ is Tonelli.
\end{itemize}
\end{prop}
 
\begin{proof} 
We aim at showing that condition \eqref{eq3 u^0_H} is equivalent to $v\leqslant 0$ on $\E(H)$. 
We first remark that any measure of the kind $\delta_{(y,0)}$ with $y\in\E$ is a Mather measure, 
i.e. $\delta_y$ is a projected Mather measure for every $y\in\E(H)$. In order to 
conclude, it suffices to show that any projected Mather measure has support contained in $\E(H)$. 
Let $\mu:={\pi_1}_{\#}\tilde\mu$ for some Mather measure $\tilde\mu\in\tilde\Mis(H)$. 
In case (i), pick $\lambda\in (0,1)$ in such a way that the function
\[
w(x):=\int_0^x \big( \lambda p_H^{-}(z)+(1-\lambda) p_H^+(z)\big)\, dz, \qquad x\in\MM
\] 
is $1$-periodic. It is easily seen that $w$ is a $C^1$-subsolution of the critical equation \eqref{eq3 H critical}  
and it is strict in $\MM\setminus\E(H)$. By using Fenchel inequality and the fact that $\tilde\mu$ is closed, we have
\begin{equation*}
-c= \int_{\TTorus} L_H(x,q)\, d \tilde{\mu}(x,q)
 \geqslant
 \int_{\TTorus} w'(x)q \, d \tilde{\mu}(x,q) 
 - 
 \int_{\T^1} H(x,w'(x)) \, d {\mu}(x)
 \geqslant
 -c,
\end{equation*}
hence all inequalities are equalities. We infer that \ $H(x,w'(x))=c$\ \ for $\mu$-a.e. $x\in\T^1$, i.e. 
$\supp(\mu)\subseteq\E(H)$.

In case (ii), it is well known, cf.  \cite[Theorem 1.6]{FSC1} that $\tilde\mu$ is invariant under the Lagrangian flow. 
Since the set $C:=\E(H)\times \{ 0 \}$ is the maximal closed invariant set under the Lagrangian flow, the 
support of $\tilde\mu$ needs to be contained in $C$. We conclude that $\supp(\mu)\subseteq\E(H)$. 
\end{proof}

\begin{oss}\label{oss3 u_0}
We refer the reader to the Appendix for an example of a non-Tonelli Hamiltonian for which condition 
\eqref{condition u_0} is violated. 
\end{oss}

\begin{teorema}\label{teo3 u_0^H}
Assume condition \eqref{condition u_0} holds.  We have
\begin{equation}\label{claim3 u_0^H}
 u^0_H(x)
 =
 \inf\{S_H(y,x)\,:\,y\in\E(H)\,\}
 \qquad
 \hbox{for all $x\in\MM$.}
\end{equation}
\end{teorema}

\begin{proof}
Let us call $u(x)$ the infimum appearing at the right-hand side of \eqref{claim3 u_0^H}. 
The fact that $u$ is a bounded solution to \eqref{eq3 H critical} can be proved as in 
Theorem \ref{teo3 bounded G sol}. The function $u$ is also $1$-periodic due to the fact that  
$z+\E(H)=\E(H)$ and $S_H(y+z,x+z)=S_H(y,x)$ for every $z\in\Z$ and $x,y\in\MM$. 
Furthermore, since $S_H(y,y)=0$ for 
every $y\in\E(G)$, it is apparent from its very definition that 
$u\leqslant 0$ on $\E(H)$. On the other hand, if $v$ is a 1-periodic 
subsolution to \eqref{eq3 H critical} with $v\leqslant 0$ on $\E(G)$, in view of Proposition 
\ref{prop S} we have
\[
 v(x)\leqslant v(y)+S_G(y,x)\leqslant S_G(y,x)
 \qquad\hbox{for all $x\in\MM$ and $y\in\E(G)$,}
\]
hence, by taking the infimum of the right-hand side term with respect to $y\in\E(G)$, 
we get $v(x)\leqslant u(x)$ for all $x\in\MM$.
\end{proof}

Let $u_G$ be the bounded solution of the critical equation \eqref{eq3 G critical}. 
Since $u_G^-:=u_G-\|u_G\|_\infty$ and 
$u_G^+:=u_G+\|u_G\|_\infty$ are, respectively, a bounded negative and positive solution 
of \eqref{eq3 G critical}, from Proposition 
\ref{prop application comparison} we derive that the functions 
$\{u^\lambda_G\,:\,\lambda>0\}$ are equi--bounded and equi--Lipschitz on $\MM$, and satisfy 
\begin{equation}\label{eq3 equibounded discounted sol}
-2\|u_G\|_\infty
\leqslant 
u^-_G(x)
\leqslant 
u^\lambda_G(x)
\leqslant 
u^+_G(x)
\leqslant 
2\|u_G\|_\infty
\qquad
\hbox{for all $x\in\MM$.}
\end{equation}
%
Let us denote by 
\[
\omegasetG:=\left\{u\in\Lip(\MM)\,: \,u^{\lambda_k}_G\ucv u\  
\hbox{in $\MM$\quad for some sequence $\lambda_k\to 0$}  \right\}.
\]
Of course, our aim is to show that this set reduces to $\{u^0_G\}$. 
We first have to identify a good candidate for $u^0_G$. Our guess is the following:
 \begin{equation}\label{def3 u_0}
 u^0_G(x):=\sup\left\{v(x)\,:\, \hbox{$v\in{\sol_b}_-(G)$}\right\}\qquad\hbox{$x\in\MM$},
\end{equation}
where ${\sol_b}_-(G)$ denotes the family of bounded 
subsolutions  $v:\MM\to 
\R$ of the critical equation \eqref{eq3 G critical} satisfying 
\begin{equation}\label{condition3 u0}
\int_\R v(y)\,d\mu(y)\leqslant \1\big(c(G)-c_f(G)\big) \qquad\text{for every $\mu\in\Mis(G)$}, 
\end{equation}
where $\Mis(G)$ denotes the set of projected Mather measures for $G$ and $\1$ is the indicator function of the set $\{0\}$ in the sense of convex analysis, i.e. $\1(t)=0$ if $t=0$ and $\1(t)=+\infty$ otherwise.  
Note that in this case the right-hand side of 
\eqref{condition3 u0} is equal to 0 since $c(G)=c_f(G)$. 

\begin{teorema}\label{teo3 u_0^G}
The following holds:
\[
 u^0_G(x)=\sup\{v(x)\,:\,\hbox{$v$ subsolution of \eqref{eq3 G critical} such that $v\leqslant 0$ on $\E(G)$}\,\}
 \qquad
 \hbox{for all $x\in\MM$.}
\]
In particular, \ 
$\displaystyle u^0_G(x)
 =
 \inf_{y\in\E(G)} S_G(y,x)$
\ {for all $x\in\MM$},\ 
and $u^0_G$ is a bounded solution to \eqref{eq3 G critical}. 
\end{teorema}

\begin{proof}
The first assertion is a straightforward consequence of 
Corollary \ref{cor Mather set=equilibria}. The second assertion can be proved arguing as in the proof 
of Theorem \ref{teo3 u_0^H}.
%
\end{proof}

Let us begin to study the asymptotic behavior of the solutions $u^\lambda_G$ of the discounted equation \eqref{eq3 G discounted}. 

\begin{prop}\label{prop3 step1}
Let $\lambda>0$. Then 
\[
u^\lambda_G(y)\leqslant 0\qquad\hbox{for all $y\in\E(G)$.}
\]
In particular, $u\leqslant u^0_G$ in $\R$ for every $u\in\omegasetG$. 
\end{prop}

\begin{proof}
Let $y\in\E(G)$. Then the stationary curve $\xi(t):=y$ for all $t\leqslant 0$ satisfies 
\[
L_G(\xi(t),\dot\xi(t))+c
=
L_G(y,0)+c
=
-\min_p G(y,p)+c
=0
\qquad
\hbox{for all $t\leqslant 0$}.
\]
By using $\xi$ as a competitor curve in formula \eqref{representation formula discounted} for $u_G^\lambda(y)$, 
we obtain $u_G^\lambda(y)\leqslant 0$. We derive that any fixed $u\in\omegasetG$ is a bounded solution of \eqref{eq3 G critical} 
satisfying $u\leqslant 0$ in $\E(G)$, hence $u\leqslant u^0_G$ in $\R$ by the maximality property stated in Theorem \ref{teo3 u_0^G}. 
\end{proof}

In order to show that $\omegasetG=\{u^G_0\}$, we need to show the converse inequality, i.e. 
\begin{equation}\label{eq3 second inequality}
 u\geqslant u^0_G\qquad\hbox{for every $u\in\omegasetG$.}
\end{equation}

Let us show the following preliminary fact first.

\begin{prop}\label{prop3 u^0_H>u^0_G}
Assume condition \eqref{condition u_0} holds. Then
\[
 u^0_H(x)\geqslant u^0_G(x)\qquad\hbox{for all $x\in\R\setminus[\underline y_V-1,\overline y_V+1]$.}
\]
\end{prop}

\begin{proof}
Let $x\in \R\setminus[\underline y_V-1,\overline y_V+1]$. According to Lemma \ref{lemma3 S_H min}, 
\[
 u^0_H(x)=\min\{S_H(y,x)\,:\,y\in\E(H)\cap [x-1,x+1]\,\}.
\]
Since $[x-1,x+1]\cap\supp(V)=\emptyset$, we have that $\E(G)\cap [x-1,x+1]=\E(H)\cap[x-1,x+1]$ and 
$S_H(y,x)=S_G(y,x)$ for all $y\in [x-1,x+1]$. We conclude that 
\begin{multline*}
  u^0_H(x)
 =
 \min\{S_H(y,x)\,:\,y\in\E(H)\cap [x-1,x+1]\,\}\\
 =
 \min\{S_G(y,x)\,:\,y\in\E(G)\cap [x-1,x+1]\,\}
 \geqslant
 \min\{S_G(y,x)\,:\,y\in\E(G)\,\}
 =u^0_G(x).
\end{multline*}
\end{proof}

We are now in position to prove the asymptotic convergence. 

\begin{teorema}\label{theo main3}
Assume condition \eqref{condition u_0} holds. 
Then the functions $u^\lambda_G$ uniformly converge to the function  
$u^0_G$ given by \eqref{def3 u_0} locally on $\MM$ as $\lambda\to 
0^+$.
\end{teorema}

\begin{proof}
Let $\lambda_n\to 0$ be such that $u^{\lambda_n}_G$ locally uniformly converge to some 
$u\in\CC(\MM)$. 
We have seen in Proposition \ref{prop3 step1} that 
$$ 
u(x)\leqslant u^0_G(x)\qquad\hbox{for every  $x\in \MM$}.
$$
To prove the opposite inequality, let us fix $x\in \MM$.  
According to Theorem \ref{teo optimal curves}, there exists a monotone curve   
$\gamma_x:(-\infty,0]\to\MM$  with $\gamma_x(0)=x$ which is optimal for $u^0_G(x)$. 
Pick $r>0$ big enough so that 
$\{x\}\cup [\underline y_V-1,\overline y_V+1]\subseteq B_r$ and let  
$\tilde\mu_{x,1}^\lambda,\,\tilde\mu_{x,2}^\lambda$ be the probability measures on $\TM$ 
defined via \eqref{def2 discounted measure1} and  \eqref{def2 discounted measure2}, respectively. 
Let us apply Proposition \ref{prop step2} with $v:=u^0_G$. 
By taking into account that 
$\supp({\tilde\mu}^\lambda_{x,2})\subseteq  \big(\R\setminus[\underline y_V-1,\overline y_V+1]\big)\times\R$ and 
$u^0_{G}\leqslant u^0_{H}$ in $\R\setminus  [\underline y_V-1,\overline y_V+1]$, we 
get 
\[
 u^\lambda_{G}(x)\geqslant u^0_G(x)
 -
 \left(
 \theta^\lambda_x\int_{\TM} u^0_G(y)\, d{\tilde\mu}^\lambda_{x,1}(y,q)
 +
 (1-\theta^\lambda_x)\int_{\TM} u^0_{H}(y)\, d{\tilde\mu}^\lambda_{x,2}(y,q)
 \right)
\]
Now set $\lambda:=\lambda_n$ and send $n\to +\infty$ in the above inequality. In view of 
Propositions \ref{prop2 precompactness} and \ref{prop2 limit measures}, there exist  
measures $\tilde\mu_{1,x}\in\parts(\TM)$ and $\tilde\mu_{2,x}\in\parts(\TTorus)$ and 
$\theta\in [0,1]$ such that
\begin{equation}\label{eq2 quasi claim}
  u(x)
  \geqslant 
  u^0_G(x)-\theta\int_{\TM} u^0_G\, d{\tilde\mu}_{x,1}
 -
 (1-\theta)\int_{\TTorus} u_H^0\, d{\tilde\mu}_{x,2}.
\end{equation}
Furthermore, in view of Propositions \ref{prop2 limit measures}, 
\[
\int_{\TM} u^0_G\, d{\tilde\mu}_{x,1}\leqslant 0\quad\hbox{if $\theta\not=0$},
\qquad\qquad 
\int_{\TTorus} u_H^0\, d{\tilde\mu}_{x,2}\leqslant 0\quad\hbox{if $\theta\not=1$.}
\]
By exploiting this information in \eqref{eq2 quasi claim}, we get $u(x)\geqslant u^0_G(x)$, 
as it was to be shown. The proof is complete. 
\end{proof}

\begin{appendix}
\section{}
In this section we give an example of a continuous Hamiltonian $H$ on $\TM$, 1--periodic in the $x$-variable, for which condition \eqref{condition u_0} does not hold. The Hamiltonian $H$ is defined on $[0,1]\times\R$ as follows: 
\[
H(x,p):=(p-p_1(x))(p-p_2(x))+\eps(x)|p-p_1(x)|\qquad\hbox{for every $x\in [0,1]$ and $p\in\R$,}
\] 
where $p_1(x):=-\cos(2\pi x)$ and $p_2(\cdot)$ is chosen in such a way that  $p_2(x)<p_1(x)$ for every $x\in (0,1)$ and 
\begin{eqnarray*}
\qquad
p_2(x)=
\begin{cases}
-1 & \hbox{for $x\in \left[0,\dfrac{1}{2}+\dfrac{2}{100}\right]\cup\left[1-\dfrac{1}{100},1\right]$},\\
\ &\ \\
p_1(x)-\eps_1 & \hbox{for $x\in \left[\dfrac{1}{2}+\dfrac{3}{100},1-\dfrac{3}{100}\right]$},
\end{cases}
\end{eqnarray*}
with $\eps_1>0$ small to be chosen later. The function $\eps(\cdot)$ is the piecewise affine function such that 
\[
\eps(0)=\eps(1)=\dfrac{1}{100}\qquad\hbox{and}\qquad \eps(x)=0\quad\hbox{for all $x\in \left[\dfrac{1}{100},1-\dfrac{1}{100} \right]$}.
\] 
We remark for later use that there exist $\rho>0$, independent of $\eps_1$, and $\eps_2\in (0,\eps_1]$  such that  
\begin{eqnarray}
&&\eps(x)+(p_1-p_2)(x)\geqslant\rho\qquad\hbox{for all $x\in \left[0,\dfrac{1}{2}+\dfrac{2}{100}\right]\cup\left[1-\dfrac{1}{100},1\right]$},\label{appendix eq1}\\
&&\eps(x)+(p_1-p_2)(x)\geqslant \eps_2\qquad\hbox{for all $x\in [0,1]$}.\nonumber
\end{eqnarray}
Let us show that $c_f(H)=c(H)=0$. To this aim, first note that the function $u(x):=\int_0^x p_1(z)\,dz$ is a 1--periodic solution to the equation 
\begin{equation}\label{appendix eq critical}
H(x,u')=0\qquad\hbox{in $\R$.}
\end{equation}
This shows that $c_f(H)\leqslant c(H)=0$.  
The converse inequality can be obtained as follows:
\[
c_f(H)=\max_x\min_p H(x,p)
\geqslant
\min_p H(0,p)
= 
\min_p \left( (p+1)^2+\dfrac{1}{100}|p+1|\right)
=
0.
\]
Furthermore, since $\min_p H(x,p)=\min_p (p-p_1(x))(p-p_2(x))<0$ for all $x\in [1/100,1-1/100]$, we derive that 
the set of equilibria $\E(H)$ is contained in $[-1/100,1/100]+\Z$.  

Let us now show that we are in case (ii) of Proposition \ref{prop3 u_0}. Indeed, the set 
\[
Z_H(x):=\{p\in\R\,:\,H(x,p)\leqslant 0\}\qquad\hbox{for all $x\in\R$}
\]
is a closed interval of the form $[p_H^-(x),p_H^+(x)]$. It is not hard to see that 
$p_H^+(x)=p_1(x)$ for all $x\in\R$. We derive 
$
\int_0^1 p^+_H(x) dx
=
\int_0^1 p_1(x) dx
=
0.
$  
In this case, any 1-periodic subsolution $v$ to the critical equation \eqref{appendix eq critical} 
satisfies $v'(x)=p
_1(x)$ for every $x\in\R$ (otherwise $v$ would not be 1-periodic). In particular, 1-periodic subsolutions to 
\eqref{appendix eq critical} are unique up to additive constants. 

Now, let $\gamma:[0,T]\to\R$ be a curve satisfying $\gamma(0)=0,\ 
\gamma(T)=1$ and 
\[
\dot\gamma(s)\in \eps(\gamma(s))+(p_1-p_2)(\gamma(s))\in\partial_p H\left(\gamma(s),p_1(\gamma(s))\right)
\quad
\hbox{for all $s\in [0,T]$.}
\]
The measure $\tilde\mu\in\parts(\TTorus)$ defined as 
\begin{equation}\label{appendix Mather measure}
\int_{\TM} f(y,q)\, d \tilde{\mu}
:=
\frac1T
\int_{0}^Tf\big(\gamma(s),\dot\gamma(s)\big)\, d s
\qquad\hbox{for every $f\in\D{C}_b(\TTorus)$}
\end{equation}
is a Mather measure. Indeed, it is closed, in the sense of Definition \ref{def closed measure} with $M:=\T^1$. Furthermore, 
\[
L(\gamma(t),\dot\gamma(t))
=
p_1(\gamma(t))\dot\gamma(t)-H(\gamma(t),p_1(t))
=
\big(v\comp\gamma\big)'(t)\qquad\hbox{for every $t\in [0,T]$}
\]
and for any 1-periodic subsolution $v$ to \eqref{appendix eq critical}, since $v'(x)=p_1(x)$ for all $x\in\R$, as remarked above. 
By integrating this equality over $[0,T]$, we get that $\tilde\mu$ is a minimizing Mather measure, cf. Theorem \ref{teo H Mather measures}. 
Let us denote by $\mu$ the projection of $\tilde\mu$ on $\T^1$ and by $u_y$, for every fixed $y\in\R$, the unique 1-periodic solution to \eqref{appendix eq critical} such that $u_y(y)=0$, i.e. the function $u_y(x):=\int_y^x p_1(z) dz$. We claim that \medskip\\
{\bf\underline{Claim:}} there exists $\eps_1>0$ small enough so that 
\[
\int_{\T^1} u_y(x)\,d\mu(x)> 0\qquad\hbox{for every $y\in \left[-\dfrac{1}{100},\dfrac{1}{100}\right]$.}
\]
In particular, for such a choice of $\eps_1$, condition \eqref{condition u_0} does not hold. 
Indeed, the maximal 1-periodic subsolution $\hat u^{\,0}$ to \eqref{appendix eq critical} satisfying $\hat u^{\,0}\leqslant 0$ on $\E(H)$ will be of the form $u_{\hat y}$ for some $\hat y\in\E(H)$. 

Let us proceed to prove the claim. We first notice that 
\[
u_y(x)\geqslant 0\qquad\hbox{for every $y\in \left[-\dfrac{1}{100},\dfrac{1}{100}\right]$ and $x\in  \left[\dfrac12+\dfrac{1}{100},1-\dfrac{1}{100}\right]$.}
\]
Furthermore, there exists $\delta>0$ small enough such that 
\[
u_y(x)\geqslant\delta\qquad
\hbox{for every $y\in \left[-\dfrac{1}{100},\dfrac{1}{100}\right]$ and $x\in  \left[\dfrac12+\dfrac{3}{100},1-\dfrac{3}{100}\right]$.}
\]
We stress that such a $\delta>0$ is independent of $\eps_1>0$.  Let $0<A<a<b<B<1$ such that 
\[
\gamma(A)=\dfrac12+\dfrac{2}{100},\quad\gamma(a)=\dfrac12+\dfrac{3}{100},
\quad \gamma(b)=1-\dfrac{3}{100},\quad \gamma(B)=1-\dfrac{1}{100}.
\]
From the definitions of $\gamma(\cdot)$ and  $p_1(\cdot)$ and from \eqref{appendix eq1} we get 
\[
\gamma(A)-\gamma(0)\geqslant \rho A,
\quad
\gamma(b)-\gamma(a)=\eps_1(b-a),
\quad  
\gamma(1)-\gamma(B)\geqslant \rho (1-B),
\]
yielding in particular $A<1/\rho$, $1-B<1/\rho$ and $b-a>1/(4\eps_1)$. 
Let us fix $y\in [-1/100,1/100]$. We have 
\begin{eqnarray*}
T\int_{\T^1} u_y\,d\,\mu
&=&\int_0^A u_y(\gamma(t))\,dt + \int_A^B u_y(\gamma(t))\,dt + \int_B^1 u_y(\gamma(t))\,dt\\
&\geqslant&
-2\|u_y\|_\infty (A+1-B) +  \int_a^b u_y(\gamma(t))\,dt
\geqslant
-\dfrac4\rho+\dfrac{\delta}{4\eps_1},
\end{eqnarray*}
where we have used the fact that $u_y(\gamma(t))\geqslant 0$ for all $t\in [A,B]$ and $\|u_y\|_\infty\leqslant 1$.  
The claim readily follows from this, by recalling that $\rho>0$ and $\delta>0$ do not depend on the choice of 
$y\in [-1/100,1/100]$ and $\eps_1>0$. 
\end{appendix}

\bibliography{discount}
\bibliographystyle{siam}

\end{document}